\documentclass{amsart}

\newtheorem{theorem}{Theorem}[section]
\newtheorem{lemma}[theorem]{Lemma}

\theoremstyle{definition}
\newtheorem{definition}[theorem]{Definition}
\newtheorem{example}[theorem]{Example}

\theoremstyle{remark}
\newtheorem{remark}[theorem]{Remark}

\numberwithin{equation}{section}

\usepackage{color}
\usepackage{amsmath}
\usepackage{amsfonts}
\usepackage{textcomp}
\usepackage{epsfig}
\usepackage{url}
\newcommand{\OM}{\Omega}
\newcommand{\RE}{\mathbb{R}}
\newcommand{\PO}{\mathbb{P}}

\newcommand{\PDif}[2]{\frac{\partial {#1}}{\partial {#2}}}
\newcommand{\Lsp}{\textrm{L}}
\newcommand{\Hsp}{\textrm{H}}
\newcommand{\Tsp}{\textrm{T}}

\newcommand{\Arr}[2]{
	\left(
		\begin{array}{c}
		{#1} 
		\\
		{#2}
		\end{array}
	\right)
}
\newcommand{\Arrtri}[3]{
	\left(
		\begin{array}{c}
		{#1} 
		\\
		{#2}
		\\
		{#3}		
		\end{array}
	\right)
}

\newcommand{\MATT}[4]{
	\left[
		\begin{array}{cc}
		{#1} 
		&
		{#2}
		\\
		{#3} 
		&
		{#4}		
		\end{array}
	\right]
}
\newcommand{\MATnine}[9]{
	\left[
		\begin{array}{ccc}
		{#1} & {#2} & {#3}
		\\
		{#4} & {#5} & {#6}
		\\
		{#7} & {#8} & {#9}		
		\end{array}
	\right]
}

\newcommand{\half}{\frac{1}{2}}

\newcommand{\average}[1]{\ensuremath{\lbrace\!\!\lbrace#1\rbrace\!\!\rbrace} } 
\newcommand{\jump}[1]{\ensuremath{[\![#1]\!]} }

\newcommand{\Th}{\mathcal{T}_h} %usata
\newcommand{\Ti}[1]{\mathcal{T}_{#1}} %usata

\newcommand{\SIG}{{\boldsymbol \sigma} }

\newcommand{\NOR}{{\boldsymbol n} }
\newcommand{\fB}{{\boldsymbol f} }
\newcommand{\vB}{{\boldsymbol v} }
\newcommand{\uB}{{\boldsymbol u} }
\newcommand{\wB}{{\boldsymbol w} }
\newcommand{\BO}[1]{{\boldsymbol #1} }
\newcommand{\lB}{{\boldsymbol \lambda} }
\newcommand{\phiB}{{\boldsymbol \varphi} }

\newcommand{\Hopt}{\mathcal{H}}

\newcommand{\Ropt}{\mathcal{R}}
\newcommand{\EPS}{ \mathcal{E} }

\newcommand{\GAM}{ \Gamma }

\newcommand{\Bnorm}[2]{\Vert {#1} \Vert_{#2} }
\newcommand{\Lnorm}[1]{{ \Vert #1 \Vert }}
\newcommand{\norm}[1]{{ \Vert #1 \Vert }}

\newcommand{\redd}[1]{{#1}}
\theoremstyle{proposition}
\newtheorem{proposition}{Proposition}
\newtheorem{corollary}[theorem]{Corollary}
\newtheorem{algorithm}[theorem]{Algorithm}

\begin{document}

\title{Analysis of Schwarz methods for a hybridizable discontinuous
  Galerkin discretization: the many subdomain case}

\author{Martin J.~Gander}
\address{Section de math\'ematiques, Universit\'e de Gen\`eve, Geneva,
Switzerland}
%% \curraddr{Department of Mathematics and Statistics,
%% Case Western Reserve University, Cleveland, Ohio 43403}
\email{martin.gander@unige.ch}

\author{Soheil Hajian}
\address{Institut f\"ur Mathematik, Humboldt-Universit\"at zu Berlin,
  Berlin, Germany}
\email{soheil.hajian@hu-berlin.de}
%% \thanks{The second author was supported by the University of Geneva
%%   during his PhD studies. The current paper was written mostly during
%%   the second authors' PhD studies.}

%    General info

\date{--}



\begin{abstract}
  Schwarz methods are attractive parallel solution techniques for
  solving large-scale linear systems obtained from discretizations of
  partial differential equations (PDEs). Due to the iterative nature
  of Schwarz methods, convergence rates are an important criterion to
  quantify their performance. Optimized Schwarz methods (OSM) form a
  class of Schwarz methods that are designed to achieve faster
  convergence rates by employing optimized transmission conditions
  between subdomains.  It has been shown recently that for a
  two-subdomain case, OSM is a natural solver for hybridizable
  discontinuous Galerkin (HDG) discretizations of elliptic PDEs. In
  this paper, we generalize the preceding result to the many-subdomain
  case and obtain sharp convergence rates with respect to the mesh
  size and polynomial degree, the subdomain diameter, and the
  zeroth-order term of the underlying PDE, which allows us for the
  first time to give precise convergence estimates for OSM used to
  solve parabolic problems by implicit time stepping. We illustrate
  our theoretical results with numerical experiments.
\end{abstract}

\maketitle
\markboth{Martin J.~Gander and Soheil Hajian}{OSM and DG}
\section{Introduction}

For the numerical treatment of a parabolic equation, e.g., 
\begin{equation}\label{eq:pde}
  \begin{array}{rcll}
    \displaystyle
    \PDif{u}{t} - \nabla \cdot ( a(x) \nabla u ) &=& 
    f(x,t)\quad & \textrm{in $\OM \times (0,T]$},\\
    u(x,t) &=& 0 & \textrm{on $\partial \OM \times (0,T]$}, \\
    u(x,0) &=& g(x) & \textrm{on $\OM$},
  \end{array}
\end{equation}
one often first discretizes the spatial dimension using a finite
difference (FD), finite element (FE) or discontinuous Galerkin (DG)
method. This approach, called {\it method of lines}, results in a {\it
  semi-discrete} system where the unknown $u(x,t)$ is approximated by
a finite dimensional vector $\uB_h(t)$ and the differential operator
$-\nabla \cdot ( a(x) \nabla )$ by a stiffness matrix which we denote
by $A_h$. More precisely we then have
\begin{equation}\label{eq:semidisc}
  \PDif{\uB_h(t)}{t} + A_h \uB_h(t) = \fB(t),
\end{equation}
and $\uB_h(t=0) = \BO{g}_h$. We then discretize in time using for
example a backward Euler method with time step $\tau$, i.e.,
\begin{equation}
\label{eq:linsystem}
	\Big( \frac{1}{\tau} M_h + A_h \Big) \uB_n =
        \frac{1}{\tau} M_h \uB_{n-1} + \fB(t_n),
\end{equation}
where $M_h$ is called the mass matrix and $\uB_n$ is an approximation
of $\uB_h(t_n)$. Therefore, at each time-step, a linear system has to be
solved. 

One approach for solving (\ref{eq:linsystem}) efficiently 
is to use a domain decomposition method where we decompose the
original spatial domain $\OM$ into overlapping or non-overlapping
subdomains and then solve smaller linear systems in parallel. In this
paper we choose the spatial discretization to be a DG method, more
precisely a hybridizable interior penalty (IPH) method.

It has been shown that optimized Schwarz methods are attractive and
natural solvers for hybridizable DG discretizations, see
\cite{hajian2014analysis,hajian2014}. This is due to the fact that
hybridizable DG methods impose continuity across elements and
subdomains using a Robin transmission condition, see
\cite{hajian2013block}. Robin transmission conditions and a
suitable choice of the Robin parameter are the core of OSM to achieve
fast convergence \cite{ganderos}. Special care is needed when OSM is
used as a solver for classical FEM when cross-points are present, see,
e.g., \cite{loisel, gander2012best, gander2013applicability,
  gander2015CPDD}. Those are points which are shared by more than two
subdomains. This is not the case when we apply OSM to a hybridizable
DG method, e.g.,~IPH, since subdomains only communicate if they have a
non-zero measure interface with each other.

We generalize here our previous results for a two subdomain
configuration in \cite{hajian2014analysis} to the case of many
subdomains, perform an analysis with respect to the polynomial degree
of the IPH, and study for the first time the influence of the
time-step $\tau$ on the performance and scalability of the OSM.
\redd{
  However we do not make an attempt to optimize the solver with
  respect to the jumps in $a(x)$ coefficient and therefore we work,
  without loss of generality in this context, with
  \begin{equation*}
    ( \eta - \Delta ) u = f \quad \text{in } \Omega,
  \end{equation*}
  where $\eta = {\tau}^{-1}$ is a constant.
}
In Section \ref{sec:iph} we recall the definition of IPH in a {\it
  hybridizable} formulation and introduce the domain decomposition
settings. In Section \ref{sec:optIPH} we introduce an OSM for IPH and
analyze its convergence properties. The main contributions of the
paper are Theorem \ref{thm:main}, Corollary \ref{cor:osm} and the
refined analysis in Section \ref{sec:sharpest}. We validate our
theoretical findings by performing numerical experiments in Section
\ref{sec:num}.
\section{The IPH method} \label{sec:iph}
\redd{
IPH was first introduced in \cite{ewing} as a stabilized discontinuous
finite element method and later was studied as a member of the class
of hybridizable DG methods in \cite{cockburn}. It has been shown that
it is equivalent to a method called Ultra Weak Variational Formulation
(UWVF) for the Helmholtz equation; see \cite{MZA:8194617}. IPH also
fits into the framework developed in \cite{dgunified} for a unified
analysis of DG methods.
}

In this section we recall the IPH method and its properties. We can
define many DG methods by two equivalent formulations, namely the {\it
  primal} and {\it flux} formulation, see for instance
\cite{dgunified}. However there is also a third equivalent formulation
for a class of hybridizable DG methods introduced in
\cite{cockburn}. For the sake of simplicity we use only the hybridized
formulation for IPH and refer the reader to
\cite{hajian2014, lehrenfeld2010hybrid} for its primal and flux formulations.

\subsection{Notation}

We now define the necessary operators and function spaces needed to
analyze DG methods. We follow the notation in \cite{dgunified}. Let
$\Th = \{ K \} $ be a shape-regular and quasi-uniform triangulation of
the domain $\OM$. We denote the diameter of an element of the
triangulation by $h_K := \max_{x,y \in K} |x-y|$ and define $h :=
\max_{K \in \Th} h_K$. If $e$ is an edge of an element, we denote the
length of that edge by $h_e$. The quasi-uniformity of the mesh implies
$h \approx h_K \approx h_e$.
Let us denote the set of interior edges shared by two elements in
$\Th$ by $\EPS^0$, i.e., $ \EPS^0 := \{ e = \partial K_1 \cap \partial
K_2, \forall K_1, K_2 \in \Th \}$. Similarly we define the set of
boundary edges by $\EPS^\partial$ and all edges by $\EPS :=
\EPS^\partial \cup \EPS^0$.

We seek a DG approximation which belongs to the finite dimensional
space
\begin{equation}
  V_h := 
  \left\lbrace v \in \Lsp^2(\OM) :
  \left. v \right|_{K} \in \PO^k(K), \forall K \in \Th \right\rbrace,
\end{equation}
where $\PO^k(K)$ is the space of polynomials of degree less than $k$
in the simplex $K \in \Th$. Note that a function in $V_h$ is not
necessarily continuous. More precisely $V_h$ is a finite dimensional
subspace of a broken Sobolev space $\Hsp^l(\Th) := \prod_{K \in \Th}
\Hsp^l(K)$, where $\Hsp^l(K)$ is the usual Sobolev space in $K \in \Th$
and $l$ is a positive integer. Since $\Hsp^l(\Th)$ contains
discontinuous functions, its trace space along $\EPS^0$ can be
double-valued. We define the trace space of functions in $\Hsp^l(\Th)$
by $\Tsp(\EPS) := \prod_{K \in \Th} \Lsp^2(\partial K)$. Observe that
$q \in \Tsp(\EPS)$ can be double-valued on $\EPS^0$ but it is
single-valued on $\EPS^\partial$.

We now define two trace operators: let $q \in \Tsp(\EPS)$ and $q_i :=
q |_{\partial K_i}$. Then on $e = \partial K_1 \cap \partial K_2$ we
define the average and jump operators
\begin{equation*}
  \begin{array}{lrlr}
    \average{q} := \frac{1}{2} ( q_1 + q_2 ), 
    & %\text{(average)}
    &
    \jump{q} := q_1 \, \NOR_1 + q_2 \, \NOR_2,
    & %\text{(jump)}		
  \end{array}
\end{equation*}
where $\NOR_i$ is the unit outward normal from $K_i$ on $e \in
\EPS^0$. Note that the jump and average definition is independent of
the element enumeration. Similarly for a vector-valued function $\SIG \in \left[
  \Tsp(\EPS) \right]^2 $ we define on interior edges
\begin{equation*}
  \begin{array}{lrlr}
    \average{\SIG} := \frac{1}{2} ( \SIG_1 + \SIG_2 ), 
    &
    % \text{(average)}
    &
    \jump{\SIG} := \SIG_1 \cdot \NOR_1 + \SIG_2 \cdot \NOR_2.
    &
    % \text{(jump)}		
  \end{array}
\end{equation*}
On the boundary, we set the average and jump operators to
$\average{\SIG} := \SIG$ and $\jump{q} = q \, \NOR$ and we do not need
to define $\average{q}$ and $\jump{\SIG}$ on $e \in \EPS^\partial$
since they do not appear in the discrete formulation.

Since $\Hsp^l(\Th)$ contains discontinuous functions, we need to define
some piecewise gradient operators. For all $u, v \in \Hsp^l(\Th)$ we define
\begin{equation*}
  \int_{\Th} \nabla u \cdot \nabla v := \sum_{K \in \Th} \int_{K}
  \nabla u \cdot \nabla v.
\end{equation*}
For $a, b \in \Tsp(\EPS)$ and single-valued on $\EPS^0$ we define the
edge integrals by
\begin{equation*}
  \int_{\EPS} a \, b  := \sum_{e \in \EPS} \int_{e} a \, b.
\end{equation*}
\subsection{Domain decomposition setting}
In order to define IPH in a hybridizable form we first decompose the
domain into $N_s$ non-overlapping subdomains $\{ \OM_i
\}_{i=1}^{N_s}$. We denote the interface between subdomains by
$\Gamma$ and assume the interface is a subset of internal edges,
$\EPS^0$. More precisely, we denote the interface between two
subdomains by $\Gamma_{ij} := \partial \OM_i \cap \partial \OM_j$ for
$i \not = j$ and the global interface by $\Gamma := \cup_{i \not= j}
\Gamma_{ij} \subset \EPS^0$. In other words the domain decomposition
does not go through any element of the triangulation. For convenience we
denote the interface belonging to subdomain $\OM_i$ by $\Gamma_i :=
\cup_{j \in N(i)} \Gamma_{ij}$ where $N(i)$ is a set containing
neighbors of the subdomain $\OM_i$, for an example see Figure
\ref{fig:ddmesh}.
\begin{figure}
  \centering
  \epsfig{file=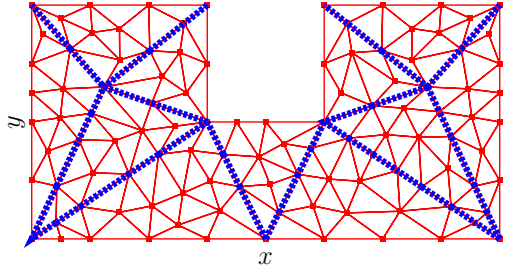, scale=0.7}
  \caption{An unstructured mesh with the interface $\Gamma$
    (blue-dashed) and cross-points.}
  \label{fig:ddmesh}
\end{figure}

This domain decomposition induces a set of non-overlapping
triangulations $\{ \Ti{i} \}_{i=1}^{N_s}$. Moreover we can define local
DG spaces on each subdomain and represent each function in $V_h$ as a
direct sum,
\begin{equation*}
  V_h = V_{h,1} \oplus V_{h,2} \oplus \hdots \oplus  V_{h,N_s},
\end{equation*}
where $V_{h,i}$ for $i=1, \hdots , N_s$ is a local space defined as
\begin{equation*}
  V_{h,i} := \big\{ v \in \Lsp^2(\OM_i) : v |_{K \in \Ti{i}} \in \PO^k(K) \big\}.
\end{equation*}
We also need a finite dimensional space on the interface which we
denote by $\Lambda_h$,
\begin{equation*}
  \Lambda_h := 
  \big\{ \varphi \in \Lsp^2(\GAM) : \left. \varphi \right|_{e \in \GAM} \in \PO^k(e) \big\}.
\end{equation*}
For the analysis of our Schwarz methods we also need to define local
spaces on $\Gamma_i$ for all $i=1,...,N_s$,
\begin{equation*}
  \Lambda_i := \big\{ \varphi \in \Lsp^2(\Gamma_i) : 
  \left. \varphi \right|_{e \in \Gamma_i} \in \PO^k(e) \big\},
\end{equation*}
and its global counterpart $\prod_{i=1}^{N_s} \Lambda_i$.  Note that
$\Lambda_h$ is single-valued across $\Gamma$ while $\prod_{i=1}^{N_s}
\Lambda_i$ is double-valued. We denote the maximum diameter of the
subdomains by $H$ and the diameter of the mono-domain $\OM$ by
$H_\OM$. We assume $0 < h \leq H < H_\OM$. For convenience we define a
function for the set of neighboring subdomains of $\Omega_i$ and
denote it by $N(i)$.

\subsection{Hybridizable formulation} 
We now present IPH in a hybridizable form. A DG method is hybridizable
if one can eliminate the degrees of freedom inside each element and
obtain a linear system in terms of a {\it single-valued} function
along edges. Not all DG methods can be written in a hybridized form,
for instance the classical IP method is not hybridizable. A hybridization
procedure for DG methods has been developed and studied in
\cite{cockburn} where IPH is also included.

In a DG context the continuity of the exact solution is imposed weakly
through a Nitsche penalization technique. Penalization is regulated
through a parameter, $\mu \in \Tsp(\EPS)$ and scaled like $\mu =
\alpha k^2/h$ for $\alpha > 0$, independent of $h$ and $k$ and
sufficiently large. This choice of $\mu$ guarantees coercivity of the
DG bilinear form and optimal approximation. Let $(u,\lambda),
(v,\varphi) \in V_h \times \Lambda_h$ and $u_i, v_i \in V_{h,i}$ be
the restriction of $u$ and $v$ to $\OM_i$. Then the IPH bilinear form
reads
\begin{equation} \label{eq:IPHbilin}
  a( (u,\lambda), (v,\varphi) ) := a_{\GAM}(\lambda,\varphi) + 
  \sum_{i=1}^{N_s} \Big( a_i(u_i,v_i) 
  + a_{i\GAM}(v_i,\lambda) + a_{i\GAM}(u_i,\varphi) \Big),
\end{equation}
where 
\begin{equation}
  \label{eq:deftildea}
    a_{\GAM}(\lambda,\varphi) := \mu \sum_{i=1}^{N_s} \int_{\Gamma_i} \lambda \, \varphi,
    \quad
    {a}_{i\GAM}(v_i,\varphi) := \int_{\Gamma_i} \Big(
    \PDif{v_i}{\NOR_i} - \mu v_i \Big) \varphi,
\end{equation}
and the local solvers $a_i(\cdot, \cdot)$ are defined as
\begin{equation} \label{eq:localsolvers}
  \begin{array}{rcl}
  a_i(u_i,v_i) &:=&  
  \displaystyle 
  \int_{\Ti{i}} \eta \, u_i \, v_i +  \nabla
  u_i \cdot \nabla v_i 
  - \int_{\EPS_i^0} \average{ \nabla u_i } \cdot \jump{v_i}
  + \average{ \nabla v_i } \cdot \jump{u_i}
  \\
  && 
  \displaystyle
  + \int_{\EPS^0_i} \frac{\mu}{2}\jump{u_i} \cdot \jump{v_i}
  - \frac{1}{2 \mu} \jump{ \nabla u_i} \jump{ \nabla v_i}
%  \\
%  &&
  \displaystyle
  + \int_{\partial \OM_i} \mu \, u_i \, v_i - \PDif{u_i}{\NOR_i} v_i - \PDif{v_i}{\NOR_i} u_i.
  \end{array}
\end{equation}
This is an IPH discretization of the model problem in $\OM_i$, and
$\partial \OM_i$ is treated as a Dirichlet boundary. Observe that
$a_i(\cdot,\cdot)$ and $a_\Gamma(\cdot,\cdot)$ are symmetric and
therefore $a(\cdot,\cdot)$ is symmetric too. 

The global bilinear form $a(\cdot,\cdot)$ is coercive at the discrete
level. In order to show coercivity we first introduce a semi-norm on
each subdomain for all $(v_i,\varphi) \in V_{h,i} \times \Lambda_h$,
\begin{equation}
    \Bnorm{(v_i,\varphi)}{i}^{2} := 
  \eta \Lnorm{ v_i }_{\Ti{i}}^{2} +
  \Lnorm{ \nabla v_i }_{\Ti{i}}^{2} + 
  \mu \Lnorm{ \jump{v_i} }_{\EPS_i \setminus  \Gamma_i}^{2}
  + \mu \Lnorm{ v_i - \varphi }_{\Gamma_i}^{2},
\end{equation}
for $i=1,..,N_s$. Note that if a subdomain is {\it floating}, that is
it does not touch the Dirichlet boundary condition, and $\eta = 0$,
then $\Bnorm{(v_i,\varphi)}{i} = 0$ implies $v_i$ and $\varphi$ are
constants and not necessarily zero. The energy norm over the whole
domain is defined by
\begin{equation}
  \Bnorm{(v,\varphi)}{}^2 := \sum_{i=1}^{N_s} \Bnorm{(v_i,\varphi)}{i}^{2}.
\end{equation}
In order to verify that this is actually a norm, we just need to check
its kernel: if $\Bnorm{(v,\varphi)}{} = 0$ then $v$ and $\varphi$ are
both constants and $ v |_{\Gamma} = \varphi $. Moreover there are
subdomains that touch the Dirichlet boundary condition and therefore
$v |_{\partial \OM} = 0$. Hence $(v,\varphi) = 0$.

The proof of coercivity for IPH is done subdomain by subdomain: we
first collect the contribution of each subdomain
\begin{equation} \label{eq:coercAtilde}
  \begin{array}{rcl}
    {a}( (v,\varphi), (v,\varphi) ) &=&
    {a}_\GAM(\varphi,\varphi) + \sum_{i=1}^{N_s}
    \big( {a}_i(v_i,v_i) + 2 {a}_{i\GAM}(v_i,\varphi) \big),
    \\ 
    &=& \sum_{i=1}^{N_s} \big( {a}_i(v_i,v_i) + 2 {a}_{i\GAM}(v_i,\varphi)
    + \mu \norm{\varphi}_{\Gamma_i}^2 \big).
  \end{array}
\end{equation}
Each right-hand side can be bounded from below by semi-norms
$\Bnorm{(v_i,\varphi)}{i}$ for $i=1,..,N_s$; for details see
\cite{hajian2014analysis,lehrenfeld2010hybrid}. Then we obtain
\begin{equation*}
  {a}( (v,\varphi), (v,\varphi) ) \geq c \sum_{i=1}^{N_s}
  \Bnorm{(v_i,\varphi)}{i}^2
  = c \Bnorm{(v,\varphi)}{}^2, \quad
  \forall (v,\varphi) \in V_h \times \Lambda_h,
\end{equation*}
where $0<c<1$ and $c$ does not depend on $h$ and $k$.

An IPH approximation of the exact solution is obtained by solving the
following problem: find $(u_h,\lambda_h) \in V_h \times \Lambda_h$
such that
\begin{equation} \label{eq:discpb}
  a( (u_h, \lambda_h), (v,\varphi) ) = \int_\OM f \, v, \quad \forall
  (v,\varphi) \in V_h \times \Lambda_h,
\end{equation}
which has a unique solution since $a(\cdot,\cdot)$ is coercive on $V_h
\times \Lambda_h$. We can also show that IPH has optimal approximation
properties, i.e., if the weak solution $u$ is regular enough then 
\begin{equation*}
%  \begin{array}{lcl}
%    \DGnorm{u_h - u}{,\ast} &\leq& c\, h^{k} | u |_{k+1,\OM},
%    \\
    \norm{u_h - u}_{0} \leq c\, h^{k+1} | u |_{k+1,\OM},
%  \end{array}
\end{equation*}
see details in \cite{dgunified, lehrenfeld2010hybrid}. 

We now describe how subdomains in the discrete problem (\ref{eq:discpb})
communicate. If we test (\ref{eq:discpb}) with $\varphi=0$ and $v=0$
in all subdomains except $\OM_i$ we obtain
\begin{equation} \label{eq:localsolve}
  a_i(u_i,v_i) + a_{i\Gamma}(v_i,\lambda_h) = \int_{\OM_i} f \, v_i,
  \quad \forall v_i \in V_{h,i}, \text{ for } i = 1, \hdots, N_s,
\end{equation}
where $u_i := u_h |_{\OM_i}$. This shows that $u_i$ is determined if
$\lambda_h$ is known. More precisely $\lambda_h$ is used as a
Dirichlet boundary data on $\partial \OM_i$ in a weak sense using a
Nitsche penalization technique. Now if we test (\ref{eq:discpb}) with
$v=0$ and $\varphi\not=0$, we obtain an equation for
$\lambda_h$:
\begin{equation} \label{eq:lambdaeq}
  a_\Gamma(\lambda_h,\varphi) + \sum_{i=1}^{N_s} a_{i\Gamma}(u_i,
  \varphi) = 0, \quad \forall \varphi \in \Lambda_h.
\end{equation}
If we further let $\varphi$ be non-zero only on $\Gamma_{ij}$, a
segment shared by $\OM_i$ and $\OM_j$, then (\ref{eq:lambdaeq}) reads
\begin{equation} \label{eq:continuity}
  \lambda_h = \frac{1}{2\mu} \Big(\mu u_i - \PDif{u_i}{\NOR_i} \Big) + 
  \frac{1}{2\mu} \Big(\mu u_j - \PDif{u_j}{\NOR_j} \Big),
  \quad \text{on} \, \GAM_{ij}.
\end{equation}  
In the language of HDG methods, equation (\ref{eq:continuity}) is called
continuity condition. The continuity condition (\ref{eq:continuity})
is the core of the optimized Schwarz method that we will describe in
Section \ref{sec:optIPH}. We have shown in \cite{hajian2014analysis}
for the case of two subdomains how to exploit (\ref{eq:continuity}) to
design a fast solver which we extend to many subdomains in this paper.

\subsection{Schur complement and matrix formulation}
The discrete problem (\ref{eq:discpb}) can be written in an equivalent
matrix form. We first choose nodal basis functions for $\PO^k(K)$ and
denote the space of degrees of freedoms (DOFs) of $V_h$ by $V$ and
similarly for subspaces, denoted by $\{V_i\}$. Then the discrete
problem (\ref{eq:discpb}) is equivalent to
\begin{equation}   \label{eq:linsyspart}
  \underbrace{
    \MATT{A_{I}}{A_{I\Gamma}}{A_{I\Gamma}^\top}{A_{\GAM}}}_{A:=}
  \Arr{\uB}{\lB} = \Arr{\fB}{0},
\end{equation}
where $\uB$ and $\lB$ are DOFs corresponding to $u_h$ and $\lambda_h$,
respectively. Here $A_I$ corresponds to the bilinear form
$\sum_{i=1}^{N_s} a_i(\cdot,\cdot)$, $A_{I\Gamma}$ corresponds to
$\sum_{i=1}^{N_s} a_{i\Gamma}(\cdot,\cdot)$ and
$A_\Gamma$ corresponds to $ a_\Gamma(\cdot,\cdot)$.

Since the bilinear form (\ref{eq:IPHbilin}) is symmetric and positive
definite we can conclude that $A$ is s.p.d. Hence its diagonal blocks,
$A_I$ and $A_\Gamma$, are also s.p.d. If we eliminate the interface
unknown, $\lB$, we arrive at a linear system in terms of {\it primal}
variables $\uB$ only. This coincides with the {\it primal} formulation
of IPH, see \cite{hajian2014analysis,lehrenfeld2010hybrid} for
details. On the other hand we can eliminate $\uB$ and obtain a Schur
complement formulation
\begin{equation} \label{eq:schureq}
  S_\Gamma \lB = \BO{g},
\end{equation}
where 
\begin{equation}
  S_\Gamma := A_\Gamma^{} - A_{I\Gamma}^\top A_I^{-1} A_{I \Gamma}^{}, 
  \quad 
  \BO{g} := - A_{I\Gamma}^\top A_I^{-1} \fB.
\end{equation}

The Schur complement matrix has smaller dimension compared to
(\ref{eq:linsyspart}) and is also s.p.d. Therefore one approach in
solving (\ref{eq:schureq}) is to use the conjugate gradient (CG)
method. However the convergence of CG is affected by the condition
number of $S_\Gamma$, which is similar to the condition number of
classical FEM Schur complement systems:
\begin{proposition} \label{prop:cond}
  Let $S_\Gamma$ be the Schur complement of the IPH discretization.
  Then for all $\varphi \in \Lambda_h$ we have
  \begin{equation}
      c \frac{H}{H_\OM^2} \norm{\varphi}_\GAM^2 \leq \phiB^\top {S}_\GAM \phiB 
      \leq C \alpha \frac{k^2}{h} \norm{\varphi}_\GAM^2,
  \end{equation}
  and therefore the condition number $\kappa( S_\Gamma )$ is bounded
  by
  \begin{equation}
    \kappa( S_\Gamma ) \leq C \, \alpha \, \frac{H_\OM^2 k^2}{H h}
    \kappa( M_\Gamma ),
  \end{equation}
  where $M_\Gamma$ is the mass matrix along the interface. Moreover
  all constants are independent of $\alpha$, $k$, $h$ and $H$.
\end{proposition}
\begin{proof}
See \cite[Appendix 3]{hajianthesis}.
\end{proof}

\section{Optimized Schwarz method for IPH} \label{sec:optIPH} 
In this section we define and analyze an optimized Schwarz method (OSM) for
IPH discretizations. Since an IPH discretization is s.p.d.~we can use
an additive Schwarz preconditioner in conjunction with CG. However it
was first observed in \cite{hajian2013block} that the convergence
mechanism of the additive Schwarz method for IPH is different from
classical FEM. For a FEM discretization, the overlap between
subdomains makes the additive Schwarz method converge, while for IPH
convergence is due to a Robin transmission condition in a
non-overlapping setting, and the Robin parameter is exactly the
penalty parameter of IPH, $\mu = \alpha \, k^2 / h$.

Robin transmission conditions are the core of OSM to obtain faster
convergence compared to the additive Schwarz method. It was shown in
\cite{ganderos} that OSM's best performance is achieved if the Robin
parameter is scaled like $1/\sqrt{h}$. This however poses a
contradiction with the IPH discretization penalty parameter since the
scaling of $\mu$ cannot be weakened otherwise coercivity and optimal
approximation properties are lost. In \cite{hajian2014analysis}, the
authors modified and analyzed an OSM while not changing the scaling of
$\mu$. They showed that for the two subdomain case the OSM's
contraction factor is $\rho \leq 1 - O(\sqrt{h})$. This is a superior
convergence factor compared to additive Schwarz with $\rho \leq 1 -
O(h)$. If OSM is used as preconditioner for a Krylov subspace method,
then a contraction factor of $\rho \leq 1 - O(h^{1/4})$ is observed.

%\marginpar{Where is the Krylov part?}
We will now define a many subdomain OSM for IPH with this property and
then analyze the solver and optimize the performance with respect to
the mesh parameter $h$ and the polynomial degree $k$.
\subsection{Definition of OSM}
We now construct an OSM for the IPH discretization. Observe that from
(\ref{eq:localsolve}), we can conclude that $u_i \in V_{h,i}$ is
determined provided $\lambda_h$ is known.
Recall also from (\ref{eq:continuity}) that two subdomains, say
$\OM_i$ and $\OM_j$ for $i\not =j$, are communicating using
\begin{equation*}
  \lambda_h = \frac{1}{2\mu} \Big(\mu u_i - \PDif{u_i}{\NOR_i} \Big) + 
  \frac{1}{2\mu} \Big(\mu u_j - \PDif{u_j}{\NOR_j} \Big),
  \quad \text{on} \, \GAM_{ij}.
\end{equation*}
Let us now assume that $\lambda_h$ is double-valued across the
interface $\Gamma$. Then we can assign an interface unknown to each
subdomain which we call $\lambda_{i}$ for $i=1,\hdots,N_s$, as
illustrated in Figure \ref{fig:lambda}.
\begin{figure}
  \centering
  \epsfig{file=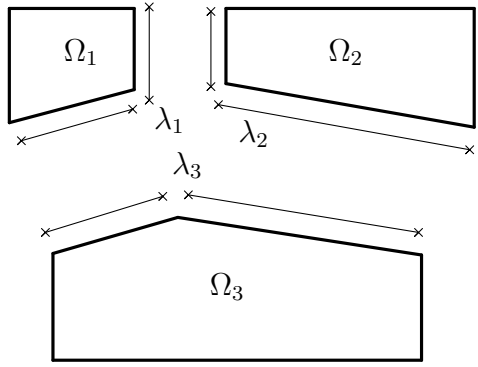, scale=1.0}
  \caption{A many subdomain configuration with unknown duplication
    along interfaces.}
  \label{fig:lambda}
\end{figure}
Therefore on each interface between two subdomains, say $\Gamma_{ij}$,
we should introduce two conditions. We do so by splitting the
continuity condition in the following fashion:
\begin{equation} \label{eq:augcont}
  \begin{array}{rcr c l}
    \gamma \lambda_{i} &+& (1-\gamma) \lambda_{j} &=&
    \frac{1}{2\mu} \Big(\mu u_i - \PDif{u_i}{\NOR_i} \Big) + 
    \frac{1}{2\mu} \Big(\mu u_j - \PDif{u_j}{\NOR_j} \Big),
    %\quad \text{on} \, \GAM_{ij},
    \\
    (1-\gamma) \lambda_{i} &+& \gamma \lambda_{j} &=&
    \frac{1}{2\mu} \Big(\mu u_i - \PDif{u_i}{\NOR_i} \Big) + 
    \frac{1}{2\mu} \Big(\mu u_j - \PDif{u_j}{\NOR_j} \Big),
    %\quad \text{on} \, \GAM_{ij},
  \end{array}
\end{equation}
where $\gamma \in \RE^+$ is a ``suitable'' parameter that we will use
to optimize the iterative method. Observe that if we subtract the two
conditions in (\ref{eq:augcont}), we arrive at $(1-2\gamma)
(\lambda_{i} - \lambda_{j}) = 0$. If $\gamma \not = \frac{1}{2}$ then
we have $\lambda_{i} = \lambda_{j} = \lambda_h$ on all $\Gamma_{ij}$,
i.e., we recover the single-valued $\lambda_h$ and therefore the
solution to the {\it augmented} system coincides with the original IPH
approximation.

The conditions in (\ref{eq:augcont}) can be written in an equivalent
variational form by multiplying them with appropriate test functions
with support on $\partial \OM_i \setminus \partial \OM$. The advantage
is that we can then use the original blocks of the IPH linear
system. For a subdomain, e.g., $\OM_i$, we obtain
\begin{equation}
  \gamma \, a_\Gamma^{(i)}(\lambda_{i}, \varphi_i) +
  a_{i\Gamma}(u_i,\varphi_i) 
  + \sum_{j \in N(i)} 2 \mu (1-\gamma) \int_{\Gamma_{ij}} \lambda_j \varphi_i 
  + \int_{\Gamma_{ij}} \Big( \PDif{u_j}{\NOR_j} - \mu u_j \Big)
  \varphi_i = 0.
\end{equation}
In order to clarify the definition of the new linear system we provide an
example in the case of two subdomains.

\begin{example}[two subdomain case] \label{example}
Suppose we have two subdomains and we call the interface between them
$\Gamma$. Then an IPH discretization with this configuration looks
like
\begin{equation} \label{eq:exampleIPH}
  \MATnine{A_1}{}{A_{1\GAM}}{}{A_2}{A_{2\GAM}}{A_{1 \GAM}^\top}{A_{2\GAM}^\top}{A_{\GAM}}
  \Arrtri{\uB_1}{\uB_2}{\lB}
  = \Arrtri{\fB_1}{\fB_2}{0}.
\end{equation}
Observe that continuity between the two subdomains is imposed through
the last row of (\ref{eq:exampleIPH}). We now suppose $\lB$ is
double-valued across the interface, $\lB_1, \lB_2$. Then we introduce two
conditions for the two interface unknowns, $\lB_1$ and $\lB_2$,
\begin{equation*}
  \begin{array}{rcrcccccc}
  \gamma A_\GAM \lB_1 &+& (1-\gamma) A_\GAM \lB_2 &+& 
  A_{1\GAM}^\top \uB_1^{} &+& A_{2\GAM}^\top \uB_2^{} &=& 0,
  \\
  (1-\gamma) A_\GAM \lB_1 &+& \gamma A_\GAM \lB_2 &+& 
  A_{1\GAM}^\top \uB_1^{} &+& A_{2\GAM}^\top \uB_2^{} &=& 0,
  \end{array}
\end{equation*}
where $\gamma \not = \frac{1}{2}$.  If we regroup the unknowns
according to the subdomain enumeration, then the ``augmented'' linear
system looks like
\begin{equation}
\label{eq:exampleAugmented}
  \left[
  \begin{array}{cc|cc}
    A_1 & A_{1\GAM} & & \\
    A_{1 \GAM}^\top & \gamma A_\GAM & A_{2 \GAM}^\top & (1-\gamma) A_\GAM \\
    \hline
    & & A_2 & A_{2\GAM}  \\
    A_{1 \GAM}^\top & (1-\gamma) A_\GAM  &  A_{2 \GAM}^\top & \gamma A_\GAM 
  \end{array}
  \right]
  \left(
  \begin{array}{c}
    \vB_1 \\
    \lB_1 \\
    \vB_2 \\
    \lB_2
  \end{array}
  \right)
  =
  \left(
  \begin{array}{c}
    \fB_1 \\
    0 \\
    \fB_2 \\
    0
  \end{array}
  \right).
\end{equation}
Note that provided $\gamma \not = \frac{1}{2}$, the linear systems
(\ref{eq:exampleIPH}) and (\ref{eq:exampleAugmented}) are equivalent
in the sense that $\vB_1=\uB_1$ and $\uB_2 = \vB_2$ and $\lB_1 = \lB_2
= \lB$.

The authors showed in \cite{hajian2014analysis} that for the
convergence analysis of a block Jacobi method applied to
(\ref{eq:exampleAugmented}) we need to obtain sharp bounds on the
eigenvalues of $A_\Gamma^{-1}B_i^{}$ where
\begin{equation} \label{eq:Boperator} \normalfont
  B_i^{} := A_{i\Gamma}^\top A_{i}^{-1} A_{i\Gamma}^{}, \quad i=1,2.
\end{equation}
Such bounds were obtained in \cite[Lemma 3.7]{hajian2014analysis}. We
will further improve the eigenvalue bounds and also obtain sharp estimates
with respect to the time step $\tau$, $\eta = \tau^{-1}$.
\end{example}

We are now in the position to define the OSM for an IPH
discretization. Formally we first construct the augmented system with
double-valued interface unknowns along the interfaces, i.e.,
(\ref{eq:exampleAugmented}). Then we rearrange the unknowns subdomain
by subdomain, i.e., collect $\{ (\uB_i, \lB_i ) \}_{i=1}^{N_s}$ and
finally we perform a block Jacobi method on the augmented linear
system with a suitable optimization parameter $\gamma$.
\begin{algorithm} \label{algo:multi}
Let $\big\{(u_i^{(0)},\lambda_{i}^{(0)})\big\}_{i=1}^{N_s}$ be a set
of initial guesses for all subdomains. Then for $n=1,2,\hdots$ find
$\big\{(u_i^{(n)},\lambda_{i}^{(n)})\big\}_{i=1}^{N_s}$ such that
\begin{equation} \label{eq:multi-loc}
    {a}_i(u_i^{(n)},v_i) +
    {a}_{i\GAM}(v_i,\lambda_{i}^{(n)} ) =
    \int_{\OM_i} f \, v_i, \quad \forall v_i \in V_{h,i},
\end{equation}
and the continuity condition on $\GAM_{ij}$ reads
\begin{equation} \label{eq:multi-cont}
  \gamma \lambda_{i}^{(n)} -
  \frac{1}{2\mu} \Big(\mu u_i - \PDif{u_i}{\NOR_i} \Big)^{(n)}
  =
  - (1- \gamma ) \lambda_{j}^{(n-1)} 
  + \frac{1}{2\mu} \Big(\mu u_j - \PDif{u_j}{\NOR_j} \Big)^{(n-1)}.
\end{equation}
\end{algorithm}
Since the solution of the augmented system coincides with the original
IPH linear system, we can conclude that Algorithm \ref{algo:multi} has
the same fixed point as the solution of the IPH discretization. 

We call $u_i^{(n)}$ satisfying (\ref{eq:multi-loc}) with $f=0$ a {\it
  discrete harmonic extension} of $\lambda_{i}$ in $\OM_i$. This
definition helps us in analyzing OSM.

\begin{definition}[Discrete harmonic extension]\label{HarmExtDef}
For all $\varphi_i \in \Lambda_i$, we denote by $\Hopt_i(\varphi_i)
\in V_{h,i}$ the discrete harmonic extension into $\OM_i$,
\begin{equation}
  \Hopt_i^{}(\varphi) \equiv - A_{i}^{-1} A_{i\GAM}^{} \phiB_i^{},
\end{equation}
where $A_i$ and $A_{i\Gamma}$ correspond to the bilinear forms
$a_i(\cdot,\cdot)$ and $a_{i\Gamma}(\cdot,\cdot)$.  The corresponding
$\varphi_i$ is called {\normalfont generator}.  In other words $u_{i}
:= \Hopt_i(\varphi_i)$ is an approximation obtained from the IPH
discretization in $\OM_i$ using $\varphi_i$ as Dirichlet data, i.e., $
A_{i} \uB_i + A_{i\GAM} \phiB_i = 0 $.
\end{definition}

There are some questions to be addressed concerning Algorithm
\ref{algo:multi}, e.g., 
\begin{enumerate}
  \item Is Algorithm \ref{algo:multi} well-posed?
  \item Does Algorithm \ref{algo:multi} converge? If yes, then can we
    obtain a contraction factor?
  \item How to use the optimization parameter $\gamma$ to improve
    the contraction factor?
  \item How do different choice of $\eta$ affect the algorithm and
its scalability?
\end{enumerate}
We will answer these questions now in Section \ref{sec:analysis}. 

\subsection{Analysis of OSM}
\label{sec:analysis}

The main goal of this section is to analyze Algorithm \ref{algo:multi}
and answer the questions regarding its well-posedness and
convergence. Our analysis is inspired by a similar result for FEM in
\cite{qin0,qin,qin1}, and we refer the reader to the original work of Lions
in \cite{lions} for an analysis at the continuous level. Our analysis is
however substantially different since DG methods impose continuity
across elements weakly. We will first prove
\begin{theorem}[Convergence estimate] \label{thm:main}
  Let the optimization parameter satisfy $\frac{1}{2} < \gamma \leq
  1$. Then Algorithm \ref{algo:multi} is well-posed and
  converges. More precisely the following contraction estimate holds
  \begin{equation*}
    \norm{ \Ropt(\varphi^{(n)}) }^2 \leq \rho \, \norm{
      \Ropt(\varphi^{(n-1)}) }^2,
  \end{equation*}
 where $\norm{ \Ropt(\varphi) }^2 := \sum_{i=1}^{N_s} \norm{
   \Ropt_i(\varphi_i) }_{\Gamma_i}^2$ and $\Ropt_i(\varphi_i) :=
 \gamma \varphi_i - \frac{1}{2\mu} \Big(\mu - \PDif{}{\NOR_i} \Big)
 \Hopt_i^{} (\varphi_i)$. Here the contraction factor, i.e., $\rho$,
 is
  \begin{equation} \label{eq:thm}
    \rho = 1 -
     \min \frac{ (2\gamma - 1) }
  { \mu (2\gamma - 1)^2 \, C(H,\eta) +1 },
  \end{equation}
where $\mu = \alpha {k^2}/{h}$ is the penalization parameter and 
\begin{equation*} \normalfont
  C(H,\eta) :=
  \left\{
  \begin{array}{ll}
    H & \text{in the case of no floating subdomains},
    \\
    \frac{1}{H \eta} & 
    \text{in the case of floating subdomains}.
  \end{array}
  \right.
\end{equation*}
\end{theorem}
% \marginpar{{How should one understand non-overlapping
% additive Schwarz here?}}
The choice $\gamma=1$ is a special case. It is shown in
\cite{hajian2014analysis, hajian2014} that in this case Algorithm
\ref{algo:multi} is equivalent to a non-overlapping additive Schwarz
method\footnote{Non-overlapping additive Schwarz method for DG methods
  means non-overlapping both at the algebraic level as well as
  continuous level in contrast to FEM.} applied to the {\it primal}
formulation of IPH. The theory for s.p.d.~preconditioners, i.e., the
abstract Schwarz framework, shows that the condition number of the
one-level additive Schwarz method for IPH is bounded by $k^2 h^{-1}
H^{-1}$. This is equivalent to a contraction factor $\rho \leq 1 -
O(\frac{h H}{k^2})$. \redd{More precisely, suppose $A$ is the original
  system matrix in primal form and $A_\text{add}$ is the corresponding
  additive Schwarz preconditioner (see for instance \cite[Section
    1.5]{widlund}) then the block Jacobi method converges with the
  aforementioned contraction factor in the $A$-norm.} It is easy to
see that our analysis also reveals the same contraction factor
\redd{(in {the} $\norm{\Ropt(\cdot)}$ norm)} in this
special case: {let} $\gamma = 1$ in (\ref{eq:thm}) and
recall that $\mu = \alpha {k^2}/{h}$. Then{,} we have
\begin{equation} \label{eq:IPHASM}
  \rho \leq 
  1 - O \Big( \frac{h H}{k^2} \Big),
\end{equation}
as $h$ and $H$ go to zero or $k$ goes to infinity. 

Our {second} objective of this section is to minimize the
contraction factor through a suitable choice of the optimization
parameter $\gamma$. This is stated in
\begin{corollary}[Optimized contraction factor] \label{cor:osm}
Let $\eta = \tau^{-1}$ where $\tau$ is the time-step which is chosen
to be $O(1)$ or $O(H)$ or $O(H^2)$. Then the optimized contraction
factor for Algorithm \ref{algo:multi} for the different choices of $\tau$
is
\begin{equation}
  \rho_\text{opt} \leq 
  \left\{
  \begin{array}{lll}
    1 - O( \frac{\sqrt{hH}}{k} ) & \text{for } \tau = O(1), & \text{if }
    \gamma_\text{opt} = \frac{1}{2} ( 1 + \frac{\sqrt{hH}}{k} ),
    \\ 
    1 - O( \frac{\sqrt{h}}{k} ) & \text{for } \tau = O(H), & \text{if }
    \gamma_\text{opt} = \frac{1}{2} ( 1 + \frac{\sqrt{h}}{k} ),
    \\
    1 - O( \sqrt{\frac{h}{H}} \frac{1}{k} ) & \text{for } \tau = O(H^2), & \text{if }
    \gamma_\text{opt} = \frac{1}{2} ( 1 + \sqrt{\frac{h}{H}} \frac{1}{k} ).
  \end{array}
  \right.
\end{equation}
\end{corollary}
Observe that the $h$-dependency and $k$-dependency is weakened by a
square-root compared to (\ref{eq:IPHASM}). Moreover if the time-step
is chosen to scale like a forward Euler time-step, i.e., $O(H^2)$, then
Algorithm \ref{algo:multi} is scalable.

{\em Proof of Theorem \ref{thm:main}.}  We first show that Algorithm
\ref{algo:multi} is well-posed, i.e., we can actually iterate. By
linearity we assume that $f=0$.  We proceed by eliminating $u_i^{(n)}$
for all subdomains and simplify Algorithm \ref{algo:multi} to: for all
subdomains, find $\lambda_i^{(n)}$ such that
\redd{
\begin{equation} \label{eq:interf-iter}
  \gamma \lambda_{i}^{(n)} -
  \frac{1}{2\mu} \Big(\mu - \PDif{}{\NOR_i} \Big) \Hopt_i^{} (\lambda_i^{(n)})
  =
  - (1- \gamma ) \lambda_{j}^{(n-1)} 
  + \frac{1}{2\mu} \Big(\mu - \PDif{}{\NOR_j} \Big)  \Hopt_j^{} (\lambda_j^{(n-1)}),
\end{equation}
}
on $\Gamma_{ij}$ for all $j \in N(i)$, where $N(i)$ is the set of
neighboring subdomains of $\OM_i$. Let us denote the linear operator
on the left-hand side by $\Ropt_i : \Lambda_i \rightarrow \Lambda_i$,
that is
\begin{equation} \label{eq:defRopt}
  \Ropt_i( \varphi_i ) := 
  \gamma \varphi_i 
  - \frac{1}{2\mu} \Big(\mu - \PDif{}{\NOR_i} \Big) \Hopt_i^{} (\varphi_i).
\end{equation}
If we show that $\Ropt_i(\cdot)$ is an invertible operator, then
Algorithm \ref{algo:multi} is well-posed. We show $\Ropt_i(\cdot)$ is
invertible by showing that it is injective:
\begin{lemma} \label{lem:injective}
If $\gamma > \frac{1}{2}$ then the operator $\Ropt_i(\cdot)$ is
injective for all $i=1,\hdots,N_s$. More precisely we have the
estimate
  \begin{equation}
    \norm{ \Ropt_i(\varphi_i) }_{\Gamma_i} \geq \Big( \gamma
    - \frac{1}{2} + c(h,H,k) \Big) \norm{ \varphi_i}_{\Gamma_i},
    \quad \forall \varphi_i \in \Lambda_i,
  \end{equation}
  where 
  \begin{equation*} \normalfont
    c(h,H,k) :=
    \left\{
    \begin{array}{ll}
      c \frac{h}{H} \frac{1}{k^2} & \text{for non-floating subdomains},
      \\
      0 & \text{for floating subdomains}.
    \end{array}
    \right.
  \end{equation*}
\end{lemma}
\begin{proof}
We multiply $\Ropt_i(\varphi_i)$ by $\varphi_i$ and integrate over
$\Gamma_i$,
\begin{equation*}
  \int_{\Gamma_i} \Ropt_i(\varphi_i) \, \varphi_i = \gamma \norm{
    \varphi_i}_{\Gamma_i}^2 + \frac{1}{2\mu} a_{i\Gamma}(u_i,\varphi_i),
\end{equation*}
where $u_i := \Hopt_i( \varphi_i )$. Recall that if $u_i$ is the
harmonic extension of $\varphi_i$ then $a_i(u_i,u_i) +
a_{i\Gamma}(u_i,\varphi_i) = 0$. Therefore we have
%\begin{equation*}
$
  \int_{\Gamma_i} \Ropt_i(\varphi_i) \, \varphi_i = \gamma \norm{
    \varphi_i}_{\Gamma_i}^2 - \frac{1}{2\mu} a_{i}(u_i,u_i).
$
%\end{equation*}
We can show that $a(u_i,u_i) \leq \big(1 - c(h,H,k) \big) \mu
\norm{\varphi_i}_{\Gamma_i}^2$, see Appendix \ref{sec:proofofestimate}, and obtain
\begin{equation*}
  \int_{\Gamma_i} \Ropt_i(\varphi_i) \, \varphi_i \geq \big( \gamma -
  \frac{1}{2} + c(h,H,k) \big) \norm{ \varphi_i}_{\Gamma_i}^2.
\end{equation*}
If $\gamma > \frac{1}{2}$, then the right-hand side is
positive. Now we apply the Cauchy-Schwarz inequality to the left-hand
side and obtain $\norm{ \Ropt_i(\varphi_i) }_{\Gamma_i} \geq \big(
\gamma - \frac{1}{2} + c(h,H,k) \big) \norm{ \varphi_i}_{\Gamma_i}$
which completes the proof.
\end{proof}

\redd{Note that Lemma \ref{lem:injective} provides a lower bound
  { for the} norm-equivalence between $\norm{
    \Ropt_i(\cdot) }_{\Gamma_i}$ and the $\Lsp^2$-norm, i.e.,
  $\norm{\cdot}_{\Gamma_i}$. The upper bound in the norm-equivalence
  can be also obtained{, as we show in the following
    proposition.}
\begin{proposition}[Norm equivalence]
{ The two norms $\norm{\Ropt_i(\cdot)}_{\Gamma_i}$ and
  $\norm{\cdot}_{\Gamma_i}$ are equivalent,}
  \begin{equation*}
    C \norm{\varphi_i}_{\Gamma_i} \geq
    \norm{ \Ropt_i(\varphi_i) }_{\Gamma_i} \geq \Big( \gamma
    - \frac{1}{2} + c(h,H,k) \Big) \norm{ \varphi_i}_{\Gamma_i},
    \quad \forall \varphi_i \in \Lambda_i,
  \end{equation*}
  where $C>0$ is independent of $h, H, \alpha$ and $\eta$. Here
  $c(h,H,k)$ is the constant defined in Lemma \ref{lem:injective}.
\end{proposition}
\begin{proof}
  The lower bound estimate is from Lemma \ref{lem:injective}. For the
  upper bound we use the estimate from Lemma \ref{lem:RtoB} (which
  will appear in Section \ref{sec:sharpest}). More precisely we have
  \begin{equation*}
    \phiB_i^\top B_i^{} \phiB_i^{} \geq c \mu
    \norm{\Ropt_i(\varphi_i)}_{\Gamma_i}^2,
  \end{equation*}
  where $B_i := A_{i\Gamma}^{\top} A_i^{-1} A_{i\Gamma}^{}$ (see
  Example \ref{example}). We then use the estimate for the eigenvalues
  of $B_i$, i.e., \cite[Lemma 3.7]{hajian2014analysis} to obtain
  \begin{equation*}
     \mu \norm{\varphi_i}_{\Gamma_i}^2 \geq 
    \phiB_i^\top B_i^{} \phiB_i^{} \geq c \mu
    \norm{\Ropt_i(\varphi_i)}_{\Gamma_i}^2.
  \end{equation*}
  This completes the proof.
\end{proof}
}

Since $\Ropt_i(\cdot)$ is linear and injective we conclude that it
induces a local norm on $\Lambda_i$. We can also define a global norm
on the space of $\prod_{i=1}^{N_s} \Lambda_i$ by
\begin{equation} \label{eq:globNorm}
  \norm{ \Ropt({\varphi}) }^2 := 
  \sum_{i=1}^{N_s} \norm{ \Ropt_i(\varphi_i) }_{\Gamma_i}^2, 
  \quad \forall {\varphi} \in \prod_{i=1}^{N_s} \Lambda_i,
\end{equation}
where ${\varphi} := (\varphi_1, \varphi_2, \hdots,
\varphi_{N_s})$. This turns out to be the right norm for the
convergence analysis of  Algorithm \ref{algo:multi}.

We can now show that Algorithm \ref{algo:multi} converges with a
concrete contraction factor estimate.  The right-hand side of the
iteration equation (\ref{eq:interf-iter}) can be simplified to
\begin{equation}
  \Ropt_i^{}(\varphi_i^{(n)}) = (2\gamma-1) \varphi_j^{(n-1)} -
  \Ropt_j^{}(\varphi_j^{(n-1)}),
\end{equation}
on $\Gamma_{ij}$ for all $j \in N(i)$. Note that $2\gamma - 1$ is
strictly-positive with our condition $\gamma > \frac{1}{2}$. For a
given subdomain, say $\OM_i$, we take the $\Lsp^2$-norm on both
sides. To simplify the presentation, we suppress the iteration index
for the moment, but terms on the left-hand side are evaluated at
iteration $(n)$ while on the right-hand side they are evaluated at
iteration index $(n-1)$:
\begin{equation*}
  \begin{array}{rcl}
    %% \overbrace{\hspace{2cm}}^\text{at iteration $(n)$}
    %% &&
    %% \overbrace{\hspace{9cm}}^\text{at iteration $(n-1)$}
    %% \\
  \norm{ \Ropt_i(\varphi_i) }_{\Gamma_{ij}}^2 &=&
  \norm{ \Ropt_j(\varphi_j) - (2\gamma-1) \varphi_j }_{\Gamma_{ij}}^2
  \\ 
  &=& \norm{ \Ropt_j(\varphi_j) }_{\Gamma_{ij}}^2 + \norm{
    (2\gamma-1)\varphi_j }_{\Gamma_{ij}}^2 -
  2 (2\gamma-1) \int_{\Gamma_{ij}} \Ropt_j(\varphi_j) \, \varphi_j
  \\
  &=& \norm{ \Ropt_j(\varphi_j) }_{\Gamma_{ij}}^2 + 
  \big[ (2\gamma-1)^2 - 2(2\gamma-1)\gamma \big] \norm{ \varphi_j
  }_{\Gamma_{ij}}^2 
  \\
  && \qquad \qquad \quad
  + \frac{1}{\mu} (2\gamma-1) \int_{\Gamma_{ij}} \big( \mu u_j -
  \PDif{u_j}{\NOR_j} \big) \varphi_j
  \\
  &=& \norm{ \Ropt_j(\varphi_j) }_{\Gamma_{ij}}^2 -
  (2\gamma-1) \Big[ \norm{ \varphi_j
  }_{\Gamma_{ij}}^2 
  - \frac{1}{\mu} \int_{\Gamma_{ij}} \big( \mu u_j -
  \PDif{u_j}{\NOR_j} \big) \varphi_j \Big].
  \end{array}
\end{equation*}
Then we sum over all interfaces of $\Omega_i$ and all subdomains to obtain
\begin{equation} \label{eq:proof}
  \begin{array}{rcl}
    \norm{ \Ropt(\varphi) }^2   &=&  
    \sum_{i=1}^{N_s} \sum_{j \in N(i)} 
  \norm{ \Ropt_j(\varphi_j) }_{\Gamma_{ij}}^2 
  \\
  && \qquad \qquad \quad 
  -  (2\gamma-1) \mu^{-1} \Big[ \mu \norm{ \varphi_j
  }_{\Gamma_{ij}}^2 
  -  \int_{\Gamma_{ij}} \big( \mu u_j -
  \PDif{u_j}{\NOR_j} \big) \varphi_j \Big]
  \\
  &=&
  \norm{ \Ropt(\varphi) }^2 
  \\
  && \qquad
  - (2\gamma-1) \mu^{-1} \sum_{m=1}^{N_s} \Big[ \mu \norm{ \varphi_m}_{\Gamma_m}^2 
  +  a_{m\Gamma}(u_m,\varphi_m) \Big]
  \\
  &=&
  \norm{ \Ropt(\varphi) }^2 
  \\
  && \qquad
  - (2\gamma-1) \mu^{-1} \sum_{m=1}^{N_s} \Big[ \mu \norm{ \varphi_m }_{\Gamma_m}^2 
  - a_{m}(u_m,u_m) \Big]
  \\
  &\leq&
  \norm{ \Ropt(\varphi) }^2 
  - c (2\gamma-1) \mu^{-1} \sum_{m=1}^{N_s} \norm{ (u_m, \varphi_m) }_{m}^2,
  \end{array}
\end{equation}
where for the left-hand side we used
\begin{equation*}
  \sum_{i=1}^{N_s} \sum_{j \in N(i)} \norm{ \Ropt_i(\varphi_i) }_{\Gamma_{ij}}^2
  =
  \sum_{i=1}^{N_s} \norm{ \Ropt_i(\varphi_i) }_{\Gamma_i}^2
  =:
  \norm{ \Ropt(\varphi) }^2,
\end{equation*}
and for the right-hand side we used the coercivity inequality
\begin{equation*}
  \mu \norm{ \varphi_m }_{\Gamma_m}^2 - a_m(u_m,u_m) \geq c 
  \norm{(u_m,\varphi_m )}_m^2,
\end{equation*}
see Appendix \ref{sec:proofofestimate} for details.  Note that
$\norm{(u_m,\varphi_m )}_m$ is {\it subdomain-wise} positive definite
if $\eta>0$. More precisely we can show that if $\eta>0$ then for all
subdomains, even {\it floating} ones, we have the estimate
\begin{equation} \label{eq:equiv}
  \norm{ \Ropt_m(\varphi_m) }_{\Gamma_m}^2 \leq 
  \Big( (2 \gamma - 1)^2 C(H,\eta) + \mu^{-1} \Big)
  \norm{ (u_m, \varphi_m) }_{m}^2,
\end{equation}
where 
\begin{equation}
  C(H,\eta) :=
  \left\{
  \begin{array}{ll}
    H & \text{for non-floating subdomain},
    \\
    \frac{1}{H \eta} & 
    \text{for floating subdomain}.
  \end{array}
  \right.
\end{equation}
Note that (\ref{eq:equiv}) makes sense only if $\eta>0$ since
$\norm{(\cdot,\cdot)}_m$ is only a semi-norm for floating subdomains
if $\eta=0$ while $\norm{\Ropt_m(\cdot)}_{\Gamma_m}$ is a norm, see
Appendix \ref{sec:proofofestimate}, in particular
(\ref{eq:lemmaFLOATING}) and (\ref{eq:lemmaNONFLOATING}). We have
ignored the $\eta \norm{u_i}_{\OM_i}$ term in
(\ref{eq:lemmaNONFLOATING}) for simplicity of the exposition; the $\eta
\norm{u_i}_{\OM_i}$ term in (\ref{eq:lemmaNONFLOATING}) will be
exploited in Section \ref{sec:sharpest}.

We then insert the norm estimate (\ref{eq:equiv}) into the last
inequality of (\ref{eq:proof}) and reintroduce the iteration index to
obtain 
\begin{equation} \label{eq:contraction}
  \norm{ \Ropt(\varphi^{(n)}) }^2 \leq
  \Big( 
  1 - \min \Big\{ \frac{2\gamma-1}{\mu(2\gamma-1)^2 H + 1},
    \frac{2\gamma-1}{\mu(2\gamma-1)^2 (H\eta)^{-1} + 1} \Big\}
  \Big)
  \norm{ \Ropt(\varphi^{(n-1)}) }^2,
\end{equation}
which shows convergence and proves Theorem \ref{thm:main}.

{\em Proof of Corollary \ref{cor:osm}.}
We need to choose a suitable $\gamma > \frac{1}{2}$ to achieve the
best possible contraction factor. In order to weaken dependencies on
the mesh parameter, subdomain diameter and polynomial degree, we make
for the optimization parameter the ansatz
\begin{equation}
  \gamma = \frac{1}{2} \Big( 1 + \frac{h^\xi H^\zeta}{k^\psi} \Big),
\end{equation}
with $\xi, \zeta, \psi \in \RE$ to be chosen. We would like to minimize the
contraction factor, i.e., 
\begin{equation}
  \rho_\text{opt} \leq 1 -
  \max_{\xi, \zeta, \psi} \, \min 
  \Big\{ 
  \frac{ {h^\xi H^\zeta} }{
    k^{2-\psi} h^{2\xi - 1} H^{2\zeta + 1} + k^{\psi}
    }
  ,
  \frac{ {h^\xi H^\zeta} }{
    k^{2-\psi} h^{2\xi - 1} H^{2\zeta - 1} \eta^{-1} + k^{\psi}
    }
  \Big\}.
\end{equation}
\redd{
\begin{remark}[On the choice of $\gamma$] \label{remark:phat}
  It has been shown in \cite{hajian2014} and \cite[Section
    3.2]{hajianthesis} that the transmission condition between two
  {subdomains} in Algorithm \ref{algo:multi} is equivalent
  at the continuous level to
  \begin{equation*}
    \Big( (2 \gamma - 1 ) \mu \, u_1 + \PDif{u_1}{\NOR_1} \Big)^{(n)} 
    = \Big( (2 \gamma - 1 ) \mu \, u_2 + \PDif{u_2}{\NOR_1} \Big)^{(n-1)}.
  \end{equation*}
  It has been shown ({at the continuous level
    \cite{ganderos}}) that the optimal choice of the Robin parameter
  is $(2\gamma - 1)\mu = O(h^{-1/2})$. This translates to choosing
  $\gamma = \frac{1}{2}(1+\sqrt{h})$. We will show that this is also
  the optimal scaling at the discrete level. {In
    \cite{ganderos}, it has been shown that the optimal scaling of the
    Robin parameter is $O( (h L)^{-1/2})$ where $L$ is the length of
    the interface and it can be viewed as a measure of the diameter of
    a subdomain, i.e., $H$. This motivates our choice of optimization
    parameter, i.e., $\gamma$.}
\end{remark}
}

When dealing with parabolic problems, $\eta = {\tau^{-1}}$ and
$\tau$ is the time-step. Therefore it is reasonable to optimize
$\gamma$ for different choices of the time-step.

\begin{itemize}
  \item $\tau = O(1)$: we start with the dependence on the
    polynomial degree. Observe that the weakest dependence is achieved if
    we let $\psi=1$. This leads to $\rho \leq 1 - O(\frac{1}{k})$,
    which compares very favorably to (\ref{eq:IPHASM}). Now we
    consider the case where $H$ is fixed and we refine the mesh, $h
    \rightarrow 0$. Then $\xi=\frac{1}{2}$ is the optimal choice which
    yields $\rho \leq 1 - O( \frac{\sqrt{h}}{k} )$. This leads
    to a simplified bound for $\rho_\text{opt}$, namely
    \begin{equation*}
      \rho_\text{opt} \leq 1 -
      \max_{\zeta} \, \min \Big\{ 
      \frac{ H^\zeta }{ H^{2\zeta+1} + 1}
      ,
      \frac{ H^\zeta }{ H^{2\zeta-1} + 1}
      \Big\}
      O( \frac{\sqrt{h} }{k} )
      .      
    \end{equation*}
    The optimal value for $\zeta$ is therefore $\frac{1}{2}$. We thus
    obtain the optimal parameter and corresponding contraction factor
    \begin{equation}
      \gamma_\text{opt} := \frac{1}{2} \Big( 1 + \frac{ \sqrt{h H}
      }{k} \Big),
      \quad
      \rho_\text{opt} \leq 1 - O \big( \frac{ \sqrt{h H} }{k} \big),
      \quad \text{if } \tau = O(1).
    \end{equation}
  \item $\tau = O(H)$: The best parameters with respect to $k$ and $h$
    follow the same argument as before. For optimization with respect
    to $H$ we have now
    \begin{equation*}
      \rho_\text{opt} \leq 1 -
      \max_{\zeta} \, \min \Big\{ 
      \frac{ H^\zeta }{ H^{2\zeta+1} + 1}
      ,
      \frac{ H^\zeta }{ H^{2\zeta} + 1}
      \Big\}
      O( \frac{\sqrt{h} }{k} )
      .      
    \end{equation*}
    In this case we can eliminate the $H$-dependence by choosing
    $\zeta=0$. Hence we have
    \begin{equation}
      \gamma_\text{opt} := \frac{1}{2} \Big( 1 + \frac{ \sqrt{h} }{k} \Big),
      \quad
      \rho_\text{opt} \leq 1 - O \big( \frac{ \sqrt{h} }{k} \big),
      \quad \text{if } \tau = O(H).
    \end{equation}
    \item $\tau=O(H^2)$: This case is comparable to using a forward
      Euler method where $\tau$ is required to be proportional to
      $h^2$. {This is a typical constraint when dealing with parabolic
        problems and accurate trajectories in time are needed, but one
        could still take larger time steps in our setting than with
        forward Euler due to a larger constant}. We proceed as before
      by choosing the same parameters with respect to $k$ and $h$. For
      the $H$-dependence we have
    \begin{equation*}
      \rho_\text{opt} \leq 1 -
      \max_{\zeta} \, \min \Big\{ 
      \frac{ H^\zeta }{ H^{2\zeta+1} + 1}
      ,
      \frac{ H^\zeta }{ H^{2\zeta+1} + 1}
      \Big\}
      O( \frac{\sqrt{h} }{k} )
      .      
    \end{equation*}
    The optimal parameter hence is $\zeta = - \frac{1}{2}$ which
    yields
    \begin{equation}
      \gamma_\text{opt} := \frac{1}{2} \Big( 1 + \sqrt{\frac{h}{H}} \frac{1 }{k} \Big),
      \quad
      \rho_\text{opt} \leq 1 - O \big( \sqrt{\frac{h}{H}} \frac{1 }{k} \big),
      \quad \text{if } \tau = O(H^2).
    \end{equation}    

 Note that this choice of $\gamma_\text{opt}$ is still feasible since
 $h \leq H$ and therefore $\gamma_\text{opt} \leq 1$. This shows that
 the method is weakly scalable if we choose a small enough time-step,
 without the need of a coarse solver. A similar result for the
 additive Schwarz method and FEM exists, see \cite[Theorem
   4]{cai1991}.
\end{itemize}
This completes the proof of Corollary \ref{cor:osm}. 

\subsection{A refined contraction factor with respect to the time-step}
\label{sec:sharpest}
\redd{ In this section we would like to investigate the effect of the
  time-step, $\tau = \eta^{-1}$, on the contraction factor while the
  number of subdomains {is} {\it fixed}, e.g., in the case
  of two subdomains.  }This has so far not been addressed, neither in
\cite{qin0} nor in the authors' paper \cite{hajian2014analysis} which
deals with two subdomains only.

\redd{Suppose for the moment that we have two subdomains. Then as mentioned
in Example \ref{example} and proved in \cite{hajian2014analysis} the
convergence of the OSM is governed by the eigenvalues of
$A_\Gamma^{-1}B_i^{}$ where $ B_i^{} := A_{i\Gamma}^\top A_{i}^{-1}
A_{i\Gamma}^{}$. We would like to obtain eigenvalue estimates that
depend on $\eta$. This is stated in the following lemma which improves
the estimate in \cite[Lemma 3.7]{hajian2014analysis}.

\begin{lemma} \label{lem:refinedB}
  Let $B_i : = A_{i\Gamma}^\top A_{i}^{-1} A_{i\Gamma}^{}$ for $i=1,2$
  where $A_i$ and $A_{i\Gamma}$ correspond to the bilinear forms
  defined in (\ref{eq:localsolvers}) and (\ref{eq:deftildea}),
  respectively. {Then for $\eta \geq 0$ we have the
    estimate}
  \begin{equation*}
    \phiB^\top_{} B_i \phiB_{}^{} \leq
    \bigg( \frac{ 1 }{ 1 + \frac{C \eta \,  h^2}{1 + C \eta \, h^2} } \bigg)
    \Big( 1 - {c} \frac{h}{H \alpha} \Big) \mu \norm{\varphi}_\Gamma^2,
  \end{equation*}
  where $c$ and $C$ are positive constants which are independent of
  $h, H, \alpha$ and $\eta$.
\end{lemma}
\begin{proof}%[Proof of Lemma \ref{lem:refinedB}]
  Recall the definition of $A_i$ from (\ref{eq:localsolvers}), and let
  us decompose $A_i$ into the mass matrix $M_i$ and the stiffness
  matrix $K_i$,
  \begin{equation*}
    A_i := \eta M_i + K_i,
  \end{equation*}
  where $\vB_i^\top M_i \uB_i := \int_{\Omega_i} u_i \, v_i$ and $K_i$
  is defined as $\vB_{i}^\top K_i^{} \uB_{i} := a_i^{}(u_i,v_i) -
  \eta \, \vB_i^\top M_i \uB_i$. Consider now
  \begin{equation*}
    \hat{A} := \MATT{K_i + \eta M_i}{A_{i\Gamma}}{A_{i \Gamma}^\top}{\half A_\Gamma},
  \end{equation*}
  which is coercive, i.e., for all $\wB := (\uB_i, \phiB)$ we have (see
  \cite[Equation 3.6]{hajian2014analysis}), 
  \begin{equation} \label{eq:dummyETA}
    \wB^\top \hat{A} \wB \geq c \, \Bnorm{(u_i,\varphi)}{i}^2 \geq \eta
    \, \uB_i^\top M_i^{} \uB_i^{} + \frac{c}{H } \norm{ \varphi
    }_\Gamma^2,
  \end{equation}
  where the last inequality is Lemma \ref{lemma:HDGlowerforB}.
  On the other hand we can easily verify that for $u_i :=
  \Hopt_i(\varphi)$ we have
  \begin{equation} \label{eq:BtoM}
    \phiB^\top B_i \phiB = \uB_i^\top \big( K_i + \eta M_i \big) \uB_i
    \leq \big( C h^{-2} + \eta \big) \uB_i^\top M_i \uB_i,
  \end{equation}
  where we have used the fact that $ \sigma(M_i^{-1} K_i^{} ) \in [c_1, c_2
    \,h^{-2}]$, which is usual for elliptic operators, see for instance
  \cite[Theorem 3.4]{castillo}. For $\wB := (\Hopt_i(\varphi),\varphi)$,
  observing that
  \begin{equation*}
    \frac{1}{2} \phiB^\top_{} A_\Gamma^{} \phiB_{}^{} - 
    \phiB^\top_{} B_i \phiB_{}^{} = \wB_{}^\top \hat{A} \wB,
  \end{equation*}
  and using (\ref{eq:dummyETA}) we have
  \begin{equation*}
    \frac{1}{2} \phiB^\top_{} A_\Gamma^{} \phiB_{}^{} - 
    \phiB^\top_{} B_i \phiB_{}^{} \geq
    \frac{\eta}{C h^{-2} + \eta} \phiB^\top_{} B_i \phiB_{}^{}
    +  \frac{c}{H} \norm{ \varphi }_\Gamma^2.
  \end{equation*}
  Recalling that $\frac{1}{2} \phiB^\top_{} A_\Gamma^{} \phiB_{}^{} =
  \mu \norm{\varphi}_\Gamma^2$ we can conclude 
  \begin{equation} \label{eq:BoptETA}
    \bigg( \frac{ 1 }{ 1 + \frac{C \eta \,  h^2}{1 + C \eta \, h^2} } \bigg)
    \Big( 1 - {c} \frac{h}{H \alpha} \Big) \mu \norm{\varphi}_\Gamma^2
    \geq \phiB^\top_{} B_i \phiB_{}^{}.
  \end{equation}
  This completes the proof.
\end{proof}

We can use Lemma \ref{lem:refinedB} to obtain a sharper contraction
factor for the two subdomain case with respect to $\eta$. In the
following corollary, we {study} the effect of $\eta$ on
the contraction factor. We consider only the case when $\gamma = 1$
for clarity of the presentation. However it is possible to use a
combination of $\gamma$ and the time-step $\tau = \eta^{-1}$ to
optimize the contraction factor. Observe that in the following
corollary, if $\eta = O(h^{-2})$ then the contraction factor is
independent of the mesh-size.
\begin{theorem}\label{cor:refinedTWO}
  Consider the two-subdomain case and let $\gamma =
  1$. {Then the error of the interface variable satisfies
    the contraction estimate}
  \begin{equation*} \normalfont
    \norm{ \mathbf{e}_i^{(n)} }_{E_i^\top E_i^{}} \leq
    \rho(h,H,\eta) \, \norm{ \mathbf{e}_i^{(n-1)} }_{E_i^\top E_i^{}}
    \quad \text{for } i = 1,2,
  \end{equation*}
  where $E_i := A_{\Gamma}^{1/2} (1- A_\Gamma^{-1} B_i^{} )$ and 
  \begin{equation*} \normalfont
    \rho(h,H,\eta) := 
    \begin{cases}
      1 - c \frac{h}{H \alpha} & \text{for } \eta = O(1) \text{ and }
      O(h^{-1}), \\ 1 - C & \text{for } \eta = O(h^{-2}).
    \end{cases}
  \end{equation*}
\end{theorem}

\begin{proof}%[Proof of Theorem \ref{cor:refinedTWO}]
  The proof relies on the proof given in \cite[Section
    4.1]{hajian2014analysis}. In the case of the two-subdomain case
  with $\gamma = 1$ we have from \cite[Section
    4.1]{hajian2014analysis} that
  \begin{equation*} \normalfont
    \norm{ \mathbf{e}_i^{(n)} }_{E_i^\top E_i^{}} \leq
    \rho(h,H,\eta) \, \norm{ \mathbf{e}_i^{(n-1)} }_{E_i^\top E_i^{}}
    \quad \text{for } i = 1,2,
  \end{equation*}
  where 
  \begin{equation*}
    \rho(h,H,\eta) := \left[ \bigg( \frac{ 1 }{ 1 + \frac{C \eta \,
        h^2}{1 + C \eta \, h^2} } \bigg) \Big( 1 - {c} \frac{h}{H
      \alpha} \Big) \right]^2,
  \end{equation*}
  which is the {square} of the upper bound constant in
  (\ref{eq:BoptETA}) divided by $\mu$.  Choosing $\eta = O(1),
  O(h^{-1})$ and $O(h^{-2})$ completes the proof. In particular,
  observe that for $\eta = O(h^{-2})$ we have
  \begin{equation*}
    \bigg( \frac{ 1 }{ 1 + \frac{C \eta \,
        h^2}{1 + C \eta \, h^2} } \bigg) \Big( 1 - {c} \frac{h}{H
      \alpha} \Big)
    \leq 
    \Big( \frac{1 - C_2 \, h}{1 + C_3} \Big) 
    \leq
    \Big( \frac{1}{1 + C_3} \Big).
  \end{equation*}
  This shows that with a time-step of the size of a forward Euler
  method, the algorithm converges in a {\it fixed number} of iterations
  since $\frac{1}{1+C_3} < 1$ uniformly in $h$.
\end{proof}

Let us now extend the above result to the case of many non-floating
subdomains. In order to do so, we first need the following lemma that
relates {the} $\Ropt_i(\cdot)$ operator to $B_i$.
\begin{lemma} \label{lem:RtoB}
  Let $B_i := A_{i\Gamma}^\top A_{i}^{-1} A_{i\Gamma}^{}$ and
  $\mathcal{R}_i(\varphi_i) := \gamma \varphi_i - \frac{1}{2\mu}
  \Big(\mu - \PDif{}{\NOR_i} \Big) \Hopt_i^{} (\varphi_i^{})$ for
  $\Omega_i$ which is a non-floating subdomain. Then the following
  estimate holds
  \begin{equation} \label{eq:RtoB}
    c \, \mu \, \norm{\Ropt_i(\varphi_i)}_{\Gamma_i}^2
    \leq \phiB_i^\top B_i^{} \phiB_i^{},
  \end{equation}
  where $c$ is independent of $h, \alpha$ and $\eta$. Moreover let
  $u_i^{} = \Hopt_i^{} (\varphi_i^{})$, then we have
  \begin{equation} \label{eq:RtoU}
    \Big(\frac{\eta}{\eta+C \, h^{-2}} \Big)
    \cdot c\, \mu \,
    \norm{\Ropt_i(\varphi_i)}_{\Gamma_i}^2
    \leq \eta \norm{u_i}_{\OM_i}^2.
  \end{equation}
\end{lemma}
\begin{proof}%[Proof of Lemma \ref{lem:RtoB}]
  We take the $\Lsp^2$-norm of $\Ropt_i(\varphi_i)$ and use the
  triangle and Young's inequality to obtain
  \begin{equation*}
    \norm{\Ropt_i(\varphi_i)}_{\Gamma_i}^2 \leq 2 \gamma^2
    \norm{\varphi_i}_{\Gamma_i}^2 + \frac{1}{2 \mu^2}
    \norm{z_i}_{\Gamma_i}^2,
  \end{equation*}
  where $z_i := (\mu - \partial_{\NOR_i}) u_i \in \Lambda_i$. We know
  from \cite[Proposition 2.4]{hajian2014analysis} that $\BO{z}_i^{} =
  M_{\Gamma_i}^{-1} B_i^{} \phiB_i^{}$. Then we have
  \begin{equation*}
    \norm{z_i}_{\Gamma_i}^2 = \phiB_i^\top B_i M_{\Gamma_i}^{-1}
    M_{\Gamma_i}^{} M_{\Gamma_i}^{-1} B_i \phiB = \phiB_i^\top B_i
    M_{\Gamma_i}^{-1} B_i \phiB_i = \phiB_i^{\top}
    B_i^{1/2} (B_i^{1/2} M_{\Gamma_i}^{-1} B_i^{1/2}) B_i^{1/2} \phiB_i^{},
  \end{equation*}
  since $B_i$ is s.p.d. A simple calculation shows that
  $\sigma(B_i^{1/2} M_{\Gamma_i}^{-1} B_i^{1/2}) =
  \sigma(M_{\Gamma_i}^{-1} B_i^{})$. Recall that $A_{\Gamma_i} = 2 \mu
  M_{\Gamma_i}$. Then for $z_i$ we have from the eigenvalues of
  $A_{\Gamma_i}^{-1}B_i^{}$, see \cite[Equation
    3.1]{hajian2014analysis},
  \begin{equation*}
    \norm{z_i}_{\Gamma_i}^2 \leq {2\mu} \cdot
    \sigma_{\max}(A_{\Gamma_i}^{-1} B_i^{}) \cdot \phiB_i^\top B_i^{}
    \phiB_i^{} \leq 2 \mu \cdot \phiB_i^\top B_i^{} \phiB_i^{}.
  \end{equation*}
  This yields
  \begin{equation*}
    \norm{\Ropt_i(\varphi_i)}_{\Gamma_i}^2 \leq 2 \gamma^2
    \norm{\varphi_i}_{\Gamma_i}^2 + \mu^{-1} \phiB_i^\top B_i^{}
    \phiB_i^{} \leq 2 \norm{\varphi_i}_{\Gamma_i}^2 + \mu^{-1}
    \phiB_i^\top B_i^{} \phiB_i^{},
  \end{equation*}
  since $\gamma\leq 1$. The last step is to use $
  \norm{\varphi_i}_{\Gamma_i}^2 \leq c_B^{-1} \mu^{-1} \phiB_i^\top
  B_i^{} \phiB_i^{}$, i.e., the lower bound for the eigenvalues of
  $A_{\Gamma_i}^{-1} B_i^{}$, see \cite[Equation
    3.1]{hajian2014analysis} where $c_B$ is independent of
  $\eta$. Hence we proved (\ref{eq:RtoB}). Using (\ref{eq:BtoM}), we
  obtain for $u_i = \Hopt_i(\varphi_i)$
  \begin{equation*}
    \Big(\frac{\eta}{\eta+C \, h^{-2}} \Big) \cdot c\, \mu \,
    \norm{\Ropt_i(\varphi_i)}_{\Gamma_i}^2 \leq \eta
    \norm{u_i}_{\OM_i}^2.
  \end{equation*}
  This completes the proof.
\end{proof}

Lemma \ref{lem:RtoB} enables us to prove the following theorem. Note
that similar to the two-subdomain case, one can obtain a contraction
factor independent of $h$ by choosing $\eta = O(h^{-2})$.
\begin{theorem} \label{thm:nonfloatRefined}
  Suppose the number of subdomains {is} fixed and they
  consist of only non-floating subdomains. Moreover let $\gamma=1$ and
  $\eta \geq 0$, then OSM converges, {and we have the
    refined contraction estimate}
  \begin{equation*}
    \norm{\Ropt(\varphi^{(n)})}^2 \leq \Big( 1 - C_1
    \frac{\eta}{\eta+C_2 h^{-2}} - c \frac{h}{H \alpha} \Big)
    \norm{\Ropt(\varphi^{(n-1)})}^2,
  \end{equation*}
  where $ \norm{\Ropt(\varphi)}^2 := \sum_{i=1}^{N_s} \norm{
    \Ropt_i(\varphi_i)}_{\Gamma_i}^2$.
\end{theorem}
\begin{proof}%[Proof of Theorem \ref{thm:nonfloatRefined}]
We consider the case when $\gamma=1$. Inserting (\ref{eq:RtoU}) into
(\ref{eq:lemmaNONFLOATING}) and then into (\ref{eq:proof}) gives
\begin{equation*}
  \norm{\Ropt(\varphi^{(n)})}^2 \leq \Big( 1 - C_1 \frac{\eta}{\eta+C_2 h^{-2}}
  - c \frac{h}{H \alpha} \Big) \norm{\Ropt(\varphi^{(n-1)})}^2.
\end{equation*}
Note that for $\eta = O(h^{-2})$ the above estimate provides a
contraction factor independent of the mesh parameter.  This completes
the proof.
\end{proof}
}
%% We can then use a Rayleigh quotient argument and obtain an upper bound
%% for $\sigma(A_\Gamma^{-1} B_i)$. Inequality (\ref{eq:BoptETA}) is of
%% great importance since it shows the dependency of  Algorithm
%% \ref{algo:multi} on $\eta$ when the number of subdomains
%% is fixed. In particular, observe that for $\eta = O(h^{-2})$ we have
%% \begin{equation*}
%%   \phiB^\top_{} B_i \phiB_{}^{} \leq 
%%   \Big( \frac{1 - C_2 \, h}{1 + C_3} \Big) \mu \norm{\varphi}_\Gamma^2
%%   \leq
%%   \Big( \frac{1}{1 + C_3} \Big) \mu \norm{\varphi}_\Gamma^2.
%% \end{equation*}
%% This shows that with a time-step of the size of a forward Euler
%% method, the algorithm converges in a {\it fixed number} of iterations
%% since $\frac{1}{1+C_3} < 1$ uniformly in $h$.

%% We now focus on the case where we have many {non-floating}
%% subdomains. In the previous section we did not fully exploit
%% (\ref{eq:lemmaNONFLOATING}), i.e., the right-hand side $\eta
%% \norm{u_i}_{\OM_i}$ term. The main step is to establish the inequality
%% \begin{equation} %\label{eq:RtoB}
%%   c \, \mu \, \norm{\Ropt_i(\varphi_i)}_i^2 \leq \phiB^\top B_i \phiB,
%% \end{equation}
%% and then to use (\ref{eq:BtoM}) to obtain
%% \begin{equation*}
%%   \Big(\frac{\eta}{\eta+C \, h^{-2}} \Big)
%%   \cdot c\, \mu \,
%%   \norm{\Ropt_i(\varphi_i)}_i^2
%%   \leq \eta \norm{u_i}_{\OM_i}^2.
%% \end{equation*}

%% It remains to prove (\ref{eq:RtoB}). 
\section{Numerical experiments} \label{sec:num}
We now illustrate our theoretical results by performing some numerical
experiments for the model problem
\begin{equation}
  \begin{array}{rcll}
    (\eta-\Delta) u &=& f,\quad & \textrm{in $\OM$},\\
    u &=& 0, & \textrm{on $\partial \OM$},
  \end{array}
\end{equation}
where $\OM$ is either the unit square, i.e.~$\Omega=(0,1)^2$, or
the domain presented in Figure \ref{fig:ddmesh}. The interface is such
that it does not cut through any element, therefore $\GAM \subset
\EPS$. We use $\PO^k$ elements and $\alpha = c {(k+1)(k+2)}$ where
$c>0$ is a constant independent of $h$ and $k$. We choose also a
randomized initial guess for Algorithm \ref{algo:multi}.

\subsection{Dependence on the mesh size}

In \cite[Section 6.3]{hajian2014analysis}, we have already
investigated numerically the convergence behavior of OSM for IPH for a
many subdomain configuration, and we show in Table \ref{tab:conv4sub}
%% from \cite[Section 6.3]{hajian2014analysis}
that indeed for a unit square domain decomposed into 110 subdomains
(see Figure \ref{fig:ddmeshONE} left)
\begin{figure}
  \centering
  \epsfig{file=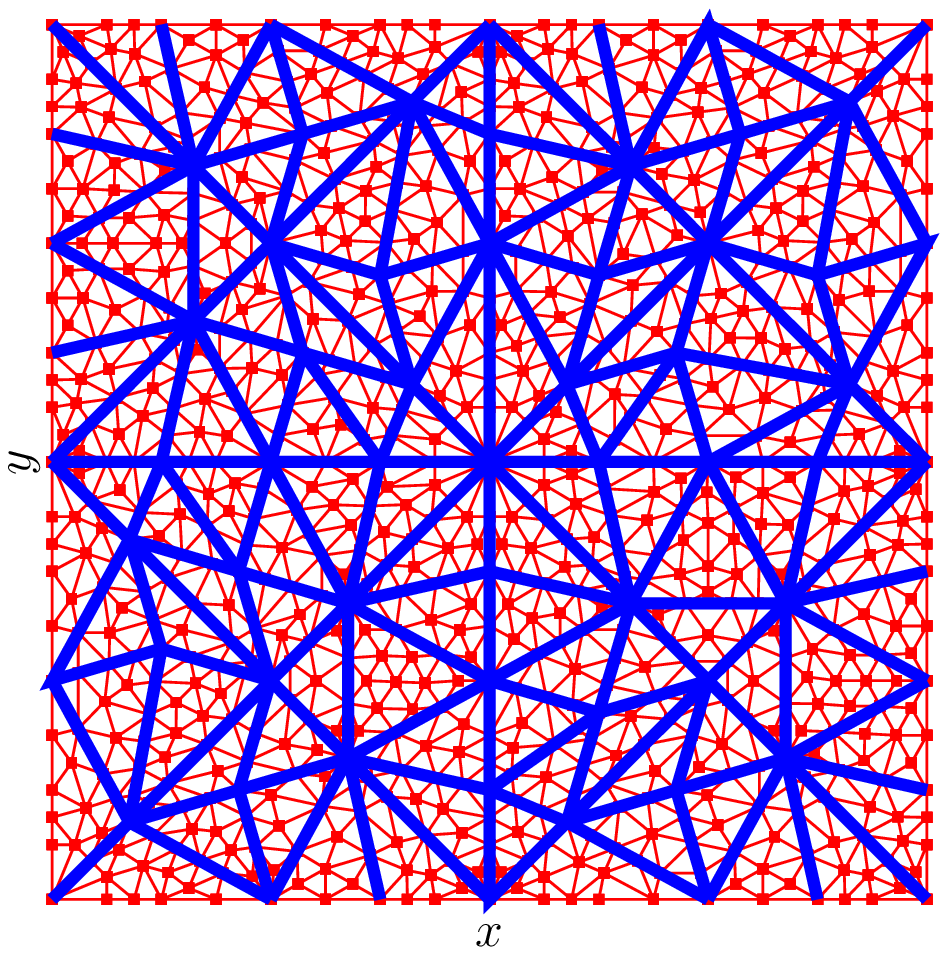, scale=0.5}
  \hspace{1cm}
  \epsfig{file=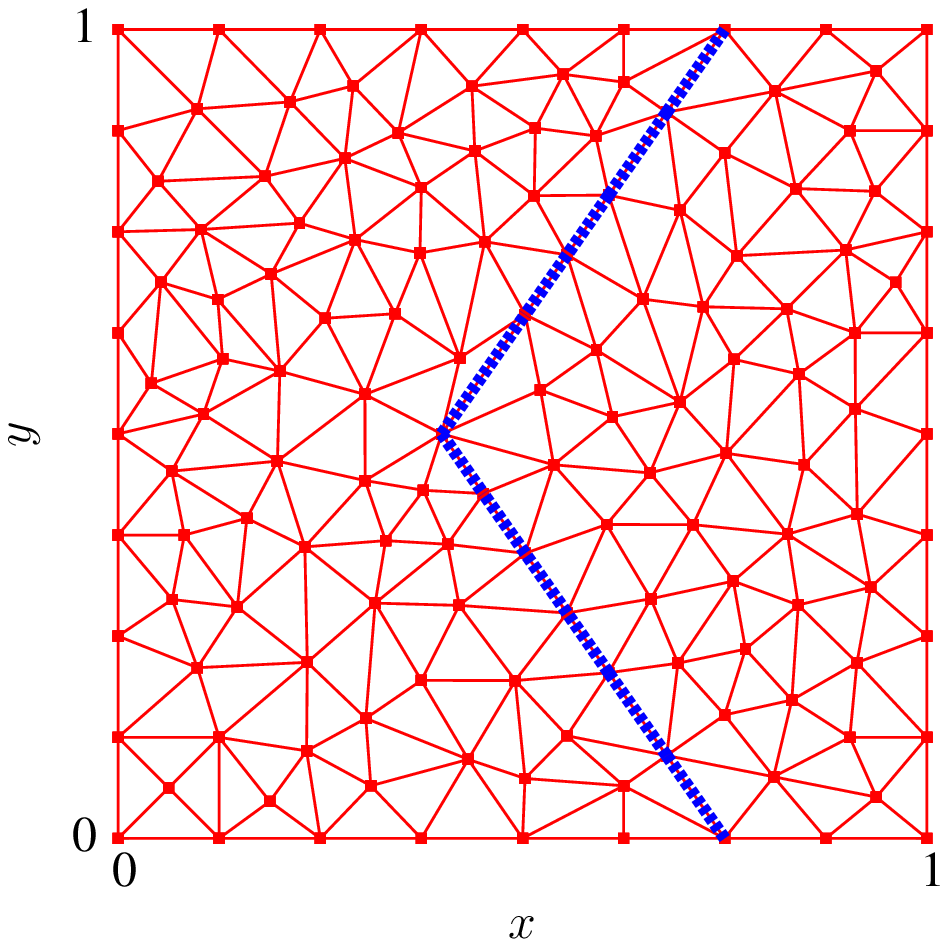, scale=0.52}
  \caption{An unstructured mesh with the interface $\Gamma$
    (blue-dashed).}
  \label{fig:ddmeshONE}
\end{figure}
the number of iterations grows like ${h}^{-1/2}$ when we refine the
mesh, provided that $\gamma = \frac{1}{2}(1 + \sqrt{h})$, as our new
theoretical analysis predicts.
%% \begin{table}
%%   \begin{tabular}{|l|ccccc|}
%%     \hline
%%     Mesh size & $h_0$ & $h_0/2$ & $h_0/4$ & $h_0/8$ & $h_0/16$
%%     % & asymptotic growth
%%     \\
%%     \hline
%%     \# iterations & 25 & 35 & 57 & 82  & 117
%%     \\ \hline
%%   \end{tabular}
%%   \caption{Convergence of OSM for four subdomains ($h$-dependence).}
%%   \centering
%%   \label{tab:conv4sub}
%% \end{table}
\begin{table}
  \begin{tabular}{|l|cccc|}
    \hline
    Mesh size & $h_0$ & $h_0/2$ & $h_0/4$ & $h_0/8$ 
    % & asymptotic growth
    \\
    \hline
    \# iterations & 1057 & 1297 & 1951 & 2734
    \\ \hline
  \end{tabular}
  \caption{Convergence of OSM for 110 subdomains ($h$-dependence).}
  \centering
  \label{tab:conv4sub}
\end{table}

\subsection{Dependence on the polynomial degree}
We next illustrate how the contraction factor of Algorithm
\ref{algo:multi} depends on the polynomial degree. First, we choose a
two subdomain configuration with a non-straight interface (see Figure
\ref{fig:ddmeshONE} right) for $\OM = (0,1)^2$. Then we choose $\gamma
= \frac{1}{2}(1+\frac{1}{k})$, $\eta = 1$ and run Algorithm
\ref{algo:multi}. We expect from our analysis to obtain $\rho \leq 1 -
O(\frac{1}{k})$, which is indeed observed in Figure %Table
\ref{fig:rhohk}.
%
%% \begin{table}
%%   \centering
%%   \begin{tabular}{|l|ccccc|}
%%     \hline
%%     Polynomial degree & 2 & 3 & 4 & 5 & 6 
%%     % & asymptotic growth
%%     \\
%%     \hline
%%     \# iterations
%%     & 446 & 555 & 685 & 812 & 938 
%%     \\ \hline
%%   \end{tabular}
%%   \caption{Convergence of OSM for two subdomains ($k$-dependence).}
%%   \label{tab:conv4subk}
%% \end{table}
%

Then we choose $\OM$ to be the domain in Figure \ref{fig:ddmesh} with
seven subdomains (including floating ones). We observe in Figure
\ref{fig:rhohk} that the number of iterations grows like
$O(k^{-1})$, which is expected from our analysis.
\begin{figure}
  \centering
  \epsfig{file=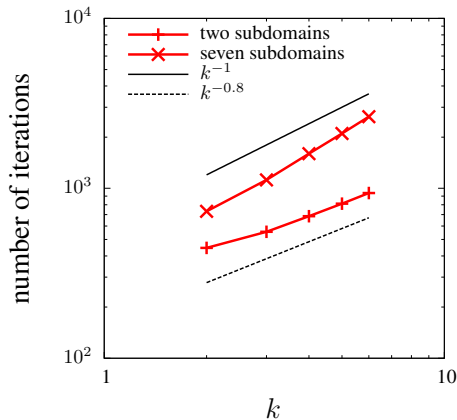, scale=0.75}
  \caption{Convergence of the OSM with respect to polynomial degree.}
  \label{fig:rhohk}
\end{figure}

% 2 732
% 3 1120
% 4 1600
% 5 2102
% 6 2640
%% \begin{table}
%%   \centering
%%   \begin{tabular}{|l|ccccc|}
%%     \hline
%%     Polynomial degree & 2 & 3 & 4 & 5 & 6 
%%     % & asymptotic growth
%%     \\
%%     \hline
%%     \# iterations
%%     & 732 & 1120 & 1600 & 2102 & 2640
%%     \\ \hline
%%   \end{tabular}
%%   \caption{Convergence of OSM for seven subdomains ($k$-dependence).}
%%   \label{tab:conv4subk1}
%% \end{table}
\subsection{Effect of the time-step on convergence}

In Section \ref{sec:sharpest} we showed how the convergence of the two
subdomain algorithm is affected by the choice of $\eta$. In Table
\ref{tab:ETAeffect} we see the number of iterations required to reach
a given accuracy for different choices of $\eta$. The domain
decomposition setting is same as Figure \ref{fig:ddmeshONE} (right).
\begin{table}
  \centering
  \begin{tabular}{|l|cccc|}
    \hline
    & $h_0$ & $h_0/2$ & $h_0/4$ & $h_0/8$ 
    % & asymptotic growth
    \\
    \hline
    case $\eta = O(1)$ & 103 & 214 & 405 & 820
    \\
    case $\eta = O(h^{-1})$ & 41 & 60 & 83 & 115
    \\ 
    case $\eta = O(h^{-2})$ & 16 & 16 & 15 & 14  
    \\ \hline
    % & $ \approx h^{-0.5}$
  \end{tabular}
  \caption{Convergence of Algorithm \ref{algo:multi} with $\gamma=1$ and
    different choices of $\eta$.}
  \label{tab:ETAeffect}
\end{table}
%

% Eta = 1, bar shape, 7 subdomain
%% k #ite

% 2 732
% 3 1120
% 4 1600
% 5 2102
% 6 2640
%
%% Eta = 1/h
%%     3.3603698    41.     41.   
%%     6.4433167    60.     60.   
%%     12.653899    83.     83.   
%%     25.383183    115.    115.  
%% Eta = 1/h^2
%%     3.3603698    16.    16.  
%%     6.4433167    16.    16.  
%%     12.653899    15.    15.  
%%     25.383183    14.    14.  
%% Eta = 1.0
%%     3.3603698    103.    103.  
%%     6.4433167    214.    214.  
%%     12.653899    405.    405.  
%%     25.383183    820.    820.  
%% Eta = 0.0
%%     3.3603698    107.    107.  
%%     6.4433167    223.    223.  
%%     12.653899    423.    423.  
%%     25.383183    842.    842.  
%
Observe that for $\eta = O(1)$, the number of iterations grows like
$O(h^{-1})$ while for $\eta = O(h^{-1})$ we observe $O(h^{-1/2})$ for
the growth of the number of iterations.  If we choose
$\eta=O(h^{-2})$, we obtain an optimal solver since the number of
iterations does not depend on the mesh parameter. 

In Table \ref{tab:compareETA} we compare the theoretical estimate in
(\ref{eq:BoptETA}) with the numerical experiments of Table
\ref{tab:ETAeffect}.
\begin{table}
  \centering
  \begin{tabular}{|l|ccc|}
    \hline
    & experiments & theoretical & 
    % & asymptotic growth
    \\
    \hline
    case $\eta = O(1)$ & $1 - h$ & $1-h$ & (sharp)
    \\
    case $\eta = O(h^{-1})$ & $1 - \sqrt{h}$ & $1-h$ & (not sharp)
    \\
    case $\eta = O(h^{-2})$ & $1-c$ & $1-c$ & (sharp)
    % & $ \approx h^{-0.5}$
    \\ \hline
  \end{tabular}
  \caption{Comparison of contraction factors between the theoretical
    estimates of Section \ref{sec:sharpest} and the numerical
    experiments.}
  \label{tab:compareETA}
\end{table}
Note that the estimates of Section \ref{sec:sharpest} can capture the
optimality of the solver when $\eta = O(h^{-2})$. However it is not
sharp when $\eta = O(h^{-1})$.

We perform the same experiment with four subdomains on $\OM = (0,1)^2$
and we choose $\eta = O(h^{-2})$. We see in Table
\ref{tab:conv4subETA} that the number of iterations remains constant as
we refine the mesh.
\begin{table}
  \begin{tabular}{|l|cccc|}
    \hline
    Mesh size & $h_0$ & $h_0/2$ & $h_0/4$ & $h_0/8$ 
    % & asymptotic growth
    \\
    \hline
    \# iterations & 144 & 157 & 168 & 164
    \\ \hline
  \end{tabular}
  \caption{Convergence of OSM for four subdomains with $\eta =
    O(h^{-2})$.} \centering
  \label{tab:conv4subETA}
\end{table}
% ETa= 1/h^2 many subdomain
%% 144
%% 157
%% 168
%% 164

{ Finally we perform numerical experiments on the weak
  scaling of the algorithm. According to Corollary \ref{cor:osm}, when
  $\tau=O(H^2)$ and the ratio $H/h$ is constant, i.e., we refine the
  mesh and the subdomain at the same time, one obtains a contraction
  factor independent of the mesh size. This can be achieved also using
  ASM applied to FEM. In Table \ref{tab:convScaling}, we illustrate
  the convergence of the OSM on a sequence of fine and coarse meshes
  such that the ratio $H/h$ remains constant.
\begin{table}
  \begin{tabular}{|l|cccc|}
    \hline
    Mesh size & $h_0$ & $h_0/2$ & $h_0/4$ & $h_0/8$ 
    % & asymptotic growth
    \\
    \hline
    \# iterations & 105 & 95 & 99 & 104
    \\ \hline
  \end{tabular}
  \caption{Convergence of OSM while the ratio $H/h$ is constant with
    $\eta = O(H^{-2})$.} \centering
  \label{tab:convScaling}
\end{table}

}

\section{Conclusion} 

We designed and analyzed an optimized Schwarz method (OSM) for the
solution of elliptic problems discretized by hybridizable interior
penalty (IPH) discontinuous Galerkin methods. Our results are a
generalization of the two subdomain analysis in
\cite{hajian2014analysis} to the case of many subdomains, and we also
study theoretically for the first time the influence of the polynomial
degree of IPH discretizations, and the effect of the time-step on the
convergence of OSM when solving parabolic problems. We derived the
optimized parameter and corresponding contraction factor for various
asymptotic regimes of the mesh and subdomain size and the time-step,
and obtained scalability without a coarse space and also mesh
independent solvers in certain specific regimes.  We validated our
theoretical results by numerical experiments. \redd{The optimized
  contraction factor shows a clear advantage of OSM compared to the
  additive Schwarz method applied to {the} primal
  formulation, e.g., see the one-level ASM version of
  \cite{karakashian} or \cite{blanca}.} The next step is to design and
analyze a coarse correction for these OSM solvers applied to IPH in
the regimes where Algorithm \ref{algo:multi} is not scalable.

\appendix

\section{Proof of some estimates} \label{sec:proofofestimate}
We now prove several technical estimates we used in the analysis of
the OSM for IPH. For all subdomains when $\eta\geq0$ we have the
inequalites
\begin{align}
  \label{eq:A2}
  \mu \norm{ \varphi_i }_{\Gamma_i}^2 - a_i(v_i,v_i) &\geq
  c \, \norm{(v_i,\varphi_i)}_i^2, 
  & 
  \forall \varphi_i \in \Lambda_i, v_i = \Hopt_i(\varphi_i),
  \\
  \label{eq:A3}
  \big( 1 - c(h,H,k) \big) \mu \norm{\varphi_i}_{\Gamma_i}^2 &\geq
  a_i(v_i,v_i),
  & 
  \forall \varphi_i \in \Lambda_i, v_i = \Hopt_i(\varphi_i),
\end{align}
where
\begin{equation*}
    c(h,H,k) :=
    \left\{
    \begin{array}{ll}
      \frac{h}{H} \frac{1}{k^2} & \text{for non-floating subdomains},
      \\
      0 & \text{for floating subdomains}.
    \end{array}
    \right.
\end{equation*}
We also have for all subdomains when $\eta>0$ the estimate
\begin{equation} \label{eq:RoptEst}
  \norm{ \Ropt_i(\varphi_i) }_{\Gamma_i}^2 \leq 
  \Big( (2 \gamma - 1)^2 C(H,\eta) + \mu^{-1} \Big)
  \norm{ (u_i, \varphi_i) }_{i}^2,
\end{equation}
where
\begin{equation*}
  C(H,\eta) :=
  \left\{
  \begin{array}{ll}
    H & \text{for non-floating subdomain},
    \\
    \frac{1}{H \eta} & 
    \text{for floating subdomain}.
  \end{array}
  \right.
\end{equation*}
We first recall an inequiality related to the coercivity of the IPH
method, that is
\begin{equation}
  \label{eq:A1}
  a_i(v_i,v_i) + 2 a_{i\Gamma}(v_i,\varphi_i) + \mu
  \norm{\varphi_i}_{\Gamma_i}^2
  \geq c \norm{(v_i,\varphi_i)}_{i}^2, 
  \quad \forall \varphi_i \in \Lambda_i, v_i \in V_{h,i}.
\end{equation}
For a proof see \cite{lehrenfeld2010hybrid,hajian2014analysis}.  The proof
of (\ref{eq:A2}) is obtained by choosing $v_i := \Hopt_i(\varphi_i)$
in (\ref{eq:A1}) and recalling the definition of the harmonic
extension which leads to $a_i(v_i,v_i) +
a_{i\Gamma}(v_i,\varphi_i)=0$. Substituting this into (\ref{eq:A1})
proves (\ref{eq:A2}).

In order to prove (\ref{eq:A3}) we decompose the proof into two parts:
floating subdomains and non-floating ones. Recall that
$\norm{(\cdot,\cdot)}_i$ is a semi-norm for floating subdomains if
$\eta = 0$, i.e., the kernel consists of constant functions. This
concludes the proof for floating subdomains with $c(h,H,k) = 0$. For
non-floating subdomains we recall a trace inequality for totally
discontinuous functions, see \cite[Lemma 3.6]{hajian2014analysis} and
\cite{brenner-poincare}:
\begin{lemma} \label{lemma:HDGlowerforB}
Let $\varphi_i \in \Lambda_i$ and $u_i \in V_{h,i}$. Let
$H_i$ be the diameter of a non-floating subdomain. Then we have
\begin{equation} \label{eq:traceineq}
  \frac{c}{H_i} \norm{ \varphi_i }_{\Gamma_i}^2 \leq \norm{ \nabla u_i
  }_{\OM_i}^2 + \mu \norm{ \jump{u_i} }_{\EPS_i \setminus \Gamma_i}^2 + \mu \norm{ u_i
    - \varphi_i }_{\Gamma_i}^2.
\end{equation}
\end{lemma}
We then substitute (\ref{eq:traceineq}) into (\ref{eq:A2}) and
recalling the definition of $\norm{(u_i,\varphi_i)}_i$ proves (\ref{eq:A3})
for non-floating subdomains with $c(h,H,k) = \frac{h}{H}\frac{1}{k^2}$.

We now prove (\ref{eq:RoptEst}). Recall that the $\Lsp^2$-norm of the
$\Ropt_i(\cdot)$ is a norm while $\norm{(\cdot,\cdot)}_i$ is only a
semi-norm for floating subdomains if $\eta=0$. Therefore
(\ref{eq:RoptEst}) makes sense for $\eta>0$. Recall the definition
of the $\Ropt_i(\cdot)$ operator,
\begin{equation*}
  \Ropt_i( \varphi_i ) := 
  \gamma \varphi_i 
  - \frac{1}{2\mu} \Big(\mu - \PDif{}{\NOR_i} \Big) u_i
  = \big( \gamma - \frac{1}{2} \big) \varphi_i 
  + 
   \frac{1}{2} \big( \varphi_i - u_i \big) + \frac{1}{2 \mu}  \PDif{u_i}{\NOR_i},
\end{equation*}
where $u_i = \Hopt_i(\varphi_i)$. We then take the $\Lsp^2$-norm over
$\Gamma_i$ and apply the triangle inequality,
\begin{equation} \label{eq:appendixPROOF}
  \begin{array}{rcl}
  \norm{ \Ropt_i(\varphi_i) }_{\Gamma_i}^2  &\leq&
  4 \big( \gamma - \frac{1}{2} \big)^2 \norm{ \varphi_i }_{\Gamma_i}^2 
  + 4 (\frac{1}{2})^2 \norm{ \varphi_i - u_i }_{\Gamma_i}^2 + 4 \big( \frac{1}{2\mu}
  \big)^2  \norm{ \PDif{u_i}{\NOR_i} }_{\Gamma_i}^2,
  \\
  &\leq& 4 \big( \gamma - \frac{1}{2} \big)^2 \norm{ \varphi_i
  }_{\Gamma_i}^2 +  \norm{ \varphi_i - u_i }_{\Gamma_i}^2 +
  c \mu^{-1} \norm{ \nabla u_i }_{\OM_i}^2,
  \\
  &\leq& 4 \big( \gamma - \frac{1}{2} \big)^2 \norm{ \varphi_i
  }_{\Gamma_i}^2 + c \mu^{-1} \big( \mu \norm{ \varphi_i - u_i }_{\Gamma_i}^2 +
   \norm{ \nabla u_i }_{\OM_i}^2 \big).
  \end{array}
\end{equation}
For non-floating subdomains we use Lemma \ref{lemma:HDGlowerforB} for the
first term on the right-hand side and obtain
\begin{equation} \label{eq:lemmaNONFLOATING}
\Big( (2\gamma - 1 )^2 H_i + c \mu^{-1} \Big) \eta \norm{u_i}_{\Omega_i}^2
+ \norm{ \Ropt_i(\varphi_i) }_{\Gamma_i}^2 \leq \Big( (2\gamma - 1 )^2
H_i + c \mu^{-1} \Big) \norm{(u_i,\varphi_i)}_i^2.
\end{equation}
For floating subdomains we use a trace inequality by Feng and
Karakashian \cite[Lemma 3.1]{karakashian},
  \begin{equation} \label{eq:karakashian}
    \norm{ u_i }_{\Gamma_i}^{2} \leq c 
    \Big[ H_i^{-1} \norm{u_i}_{\Omega_i}^{2} + H_i^{} \big( \norm{ \nabla u_i }_{\Omega_i}^{2} + 
      h^{-1} \norm{\jump{u_i}}_{\EPS_i \setminus \Gamma_i}^{2} \big)
      \Big].
  \end{equation}
We then invoke $\norm{\varphi_i}_{\Gamma_i}^2 \leq 2
\norm{u_i}_{\Gamma_i}^2 + 2 \norm{u_i - \varphi_i}_{\Gamma_i}^2$, use
(\ref{eq:karakashian}) and recall the definition of
$\norm{(u_i,\varphi_i)}$ to obtain
\begin{equation*}
  \norm{\varphi_i}_{\Gamma_i}^2 \leq \frac{C}{H_i \eta} \norm{(u_i,\varphi_i)}_i^2.
\end{equation*}
Substituting this estimate back into (\ref{eq:appendixPROOF}) yields 
\begin{equation} \label{eq:lemmaFLOATING}
  \norm{ \Ropt_i(\varphi_i) }_{\Gamma_i}^2 \leq \Big( (2\gamma - 1 )^2
  \cdot C \cdot (H_i \eta)^{-1} + c \mu^{-1} \Big)
  \norm{(u_i,\varphi_i)}_i^2.
\end{equation}

   %% \\
   %% &=& 2 \big( \gamma - \frac{1}{2} \big)^2 \norm{ \varphi_i
   %% }_{\Gamma_i}^2 + c \mu^{-1} \norm{(u_i,\varphi_i)}_i^2,
   %% \\
   %% &\leq& \big( C(H,\eta) ( 2\gamma - 1)^2 + c \mu^{-1} \big)
   %% \norm{(u_i,\varphi_i)}_i^2.

%

\bibliographystyle{amsplain}
\bibliography{osdg}
\end{document}